\documentclass[11pt,a4paper]{article}
\pdfoutput=1
\usepackage[margin=1in]{geometry}

\usepackage[utf8]{inputenc}
\usepackage[dvipsnames]{xcolor}
\usepackage[bookmarksnumbered,linktocpage,hypertexnames=false,colorlinks=true,linkcolor=NavyBlue,urlcolor=NavyBlue,citecolor=ForestGreen,anchorcolor=green,breaklinks=true,pagebackref=true,pdfusetitle]{hyperref}
\usepackage{bbm,braket,microtype,amsmath,amssymb,amsthm,amsfonts,dsfont,dynkin-diagrams,mathtools,colortbl,booktabs,algorithm,algpseudocode,graphicx,enumitem,xspace,float,mathdots,ellipsis,caption,subcaption,mleftright,multirow,adjustbox}
\usepackage[normalem]{ulem}
\usepackage[T1]{fontenc}
\usepackage{textcomp}
\usepackage[english]{babel}
\usepackage[capitalize,nameinlink]{cleveref}
\usepackage{tikz}
\usepackage{stmaryrd}
\usepackage{tikz-cd}
\usepackage{xparse}
\usepackage{thm-restate}

\newcommand{\ideal}[1]{\langle #1 \rangle}

\allowdisplaybreaks[4]
\newtheorem{theorem}{Theorem}[section]
\newtheorem{corollary}[theorem]{Corollary}\AddToHook{env/corollary/begin}{\crefalias{theorem}{corollary}}
\newtheorem{definition}[theorem]{Definition}\AddToHook{env/definition/begin}{\crefalias{theorem}{definition}}
\newtheorem{lemma}[theorem]{Lemma}\AddToHook{env/lemma/begin}{\crefalias{theorem}{lemma}}
\newtheorem{proposition}[theorem]{Proposition}\AddToHook{env/proposition/begin}{\crefalias{theorem}{proposition}}\crefname{proposition}{Proposition}{Propositions}
\AddToHook{env/problem/begin}{\crefalias{theorem}{problem}}
\AddToHook{env/assumption/begin}{\crefalias{theorem}{assumption}}
\AddToHook{env/claim/begin}{\crefalias{theorem}{claim}}
\AddToHook{env/conjecture/begin}{\crefalias{theorem}{conjecture}}
\newtheorem{remark}[theorem]{Remark}\AddToHook{env/remark/begin}{\crefalias{theorem}{remark}}
\newtheorem{example}[theorem]{Example}\AddToHook{env/example/begin}{\crefalias{theorem}{example}}
\AddToHook{env/question/begin}{\crefalias{theorem}{question}}
\newtheorem*{theorem*}{Theorem}
\newtheorem*{corollary*}{Corollary}
\newtheorem*{proposition*}{Proposition}
\AddToHook{cmd/appendix/after}{\crefalias{section}{appendix}}
\numberwithin{equation}{section}

\usepackage{aliascnt}
\makeatletter
\let\c@algorithm\relax 
\makeatother
\newaliascnt{algorithm}{theorem}

\makeatletter
\let\c@table\relax 
\makeatother
\newaliascnt{table}{theorem}

\newcommand{\PP}{\mathbb{P}}
\newcommand{\GG}{\mathbb{G}}

\newcommand{\ZZ}{\mathbb{Z}}

\newcommand{\NN}{\mathbb{Z}_{\geq 0}}
\newcommand{\KK}{\mathbb{K}}
\newcommand{\cR}{\mathcal{R}}
\newcommand{\cX}{\mathcal{X}}
\newcommand{\cP}{\mathcal{P}}

\newcommand{\cZ}{\mathcal{Z}}

\newcommand{\cH}{\mathcal{H}}
\newcommand{\cI}{\mathcal{I}}
\newcommand{\cO}{\mathcal{O}}
\newcommand{\xx}{\ensuremath{\boldsymbol{x}}}
\newcommand{\ff}{\ensuremath{\boldsymbol{f}}}
\newcommand{\GL}{\mathrm{GL}}
\newcommand{\SL}{\mathrm{SL}}

\newcommand{\dd}{\mathbf{d}}

\DeclareMathOperator{\poly}{poly}

\DeclareMathOperator{\codim}{codim}
\DeclareMathOperator{\rk}{rk}
\DeclareMathOperator{\Spec}{Spec}
\DeclareMathOperator{\Proj}{Proj}
\DeclareMathOperator{\Res}{Res}

\title{Small Resultant Systems via Linear Combinations}
\author{
M. Levent Do\u{g}an 
\thanks{Ludwig-Maximilians-Universit\"{a}t M\"{u}nchen, Germany. \url{mahmut.dogan@lmu.de} }
\and 
Elias Tsigaridas 
\thanks{Inria Paris and Sorbonne University, France.
\url{elias.tsigaridas@inria.fr}
}
\and 
Zafeirakis Zafeirakopoulos \thanks{Department of Mathematics, National and Kapodistrian University of Athens, Greece. \url{zzafeir@math.uoa.gr}}}
\date{}

\begin{document}
\maketitle
{\centering\footnotesize Dedicated to the memory of Winfried Bruns.\par}

\begin{abstract}
	For a system of $s$ homogeneous polynomials of degree $d$ in $n$ variables, say $\ff = 0$,
	we consider the problem of constructing \emph{resultant systems}.
	A resultant system is a finite set of polynomials in the coefficients of the input polynomials, the vanishing of which characterizes the systems $\ff$ with a common non-zero solution.
The classical approaches for constructing resultant systems rely either on maximal minors of large coefficient matrices or on the coefficients of a resultant of generic linear combinations of the input polynomials.
Typically, they produce resultant systems containing a very large number of polynomials.
    
We develop new constructions based on taking resultants of linear combinations of the input polynomials; 
this results in resultant systems of small cardinality.
Our main results are:
    \begin{itemize}
    \item We prove that a resultant system with $\binom{d+n-1}{n-1}s-n^2+1$ polynomials exists; each polynomial is the resultant of $n$ linear combinations of the input polynomials.
    This improves the previously known upper bounds, even for systems of bivariate homogeneous polynomials.
    \item Under the assumption that the input polynomials are non-zero,
    we construct explicit resultant systems with cardinality $\poly(s,d)$, when $n$ is fixed.
    \end{itemize}
\end{abstract}

\tableofcontents

\section{Introduction}
Resultant (systems) are central in elimination theory: they characterize when a system of polynomial equations has a solution.
Consider a system of $s$ homogeneous forms 
\begin{equation}
\label{eq:system}
\ff(\xx) = 0 \Leftrightarrow
\begin{cases}
    f_1(x_1, \dots, x_n) &\coloneqq \sum_{|\alpha|=d_1} c_{1,\alpha} \xx^{\alpha} =\,0 \\
    &\;\vdots  \\
    f_s(x_1, \dots, x_n) &\coloneqq \sum_{|\alpha|=d_s} c_{s,\alpha} \xx^{\alpha} =\,0
\end{cases} \enspace,
\end{equation}
in $n\geq 2$ variables,
each having degree $d_i\coloneqq \deg f_i$
and coefficients in an algebraically closed field $\KK$ of characteristic $0$.
We ask: 
\begin{center}
	\emph{When does the polynomial system $\ff(\xx) = 0$ in \eqref{eq:system}
	have a solution in  $\PP^{n-1}_{\KK}$?}
\end{center}

If $s< n$, then there is always a non-zero solution to \eqref{eq:system}.
So, without loss of generality, we will assume  that $s\geq n$.
Further assume that $d_1\geq d_2\geq\dots\geq d_s$.

We view the polynomials $\ff\coloneqq(f_1,\ldots,f_s)$ as a point in the affine space
\[
	\cP^{d_1,\ldots,d_s} := \prod\nolimits_{i=1}^s \KK[\xx]_{d_i};
\]
after choosing the standard monomial basis, this is equivalently a tuple of coefficient vectors.
Here, $\KK[\xx]_{d_i}$ denotes the vector space of homogeneous forms 
in $n$ variables of degree $d_i$. 
A standard incidence-variety argument,
see~\cref{prop:resultant_variety}, shows that the set of tuples 
$\ff\in\cP^{d_1,\ldots,d_s}$ admitting a common non-zero solution 
is an irreducible algebraic variety
\[
	\cR_{d_1,\ldots,d_s} \subseteq \cP^{d_1,\ldots,d_s};
\]
we call it the \emph{resultant variety}. Its codimension is $s-n+1$.

Consequently, there exists a finite collection of polynomials 
in the coefficients of $f_1,\ldots,f_s$ whose simultaneous vanishing
characterizes the existence of a common non-zero solution. 
When $s=n$, the resultant variety has codimension one and 
is a hypersurface: it is defined by a single polynomial;
this is the classical Macaulay resultant. 
When $s>n$, its codimension is greater than one, so, in general,
we need several polynomials to describe the resultant variety.

For every $i \in[s]$ and  multi-index $\alpha\in\ZZ_{\geq 0}^n$
with $|\alpha|=d_i$, let $y_{i,\alpha}$ be a variable corresponding
to the coefficient of the monomial $\xx^\alpha$ in $f_i$.
In this setting, for
$R \in \KK\bigl[y_{i,\alpha} \,:\, i\in[s],\ |\alpha| = d_i ]$
and a set of polynomials $\ff=(f_1,\ldots,f_s)$,
where $f_i=\sum_{|\alpha|=d_i}c_{i,\alpha}\xx^\alpha$,
we consider the evaluation
$R(\ff)\coloneqq R\bigl((c_{i,\alpha})_{i,\alpha}\bigr)$;
that is the specialization of the polynomials of the resultant system
to the coefficients of the input polynomials $f_i$.

The previous discussion motivates the following definition.

\begin{definition}[Resultant system]
	\label{def:resultant-system}
A finite collection of polynomials
\[
	R_1,\ldots,R_N
	\subset
	\KK\big[y_{i,\alpha}:i\in[s],\ |\alpha|=d_i\big]
\]
is called a \emph{resultant system} for $s$ homogeneous forms 
of degrees $d_1,\ldots,d_s$ in $n$ variables if, 
for every $\ff=(f_1,\ldots,f_s)\in\cP^{d_1,\ldots,d_s}$, we have
\[
	\cZ_{\PP}(f_1,\ldots,f_s)\neq\varnothing
	\quad\Longleftrightarrow\quad
	R_1(\ff)=\cdots=R_N(\ff)=0.
\]
\end{definition}

In other words, a resultant system is a 
collection of polynomials whose common zero set 
in the coefficient space is exactly the resultant variety, \emph{set-theoretically}. 
The classical constructions of resultant systems 
may produce a very large number of such polynomials, 
either as maximal minors of coefficient matrices or 
as the coefficients of a resultant involving (a suitable number of) 
generic linear combinations.
It is therefore natural to ask what the
smallest sufficient number of these polynomials is.
Our main goal is to construct resultant systems 
containing a \emph{small} number of polynomials,
that in turn we compute as classical resultants of suitable linear combinations of the original forms.

\begin{example}
\label{ex:maximal-minors}
    Suppose $d_1 = \dots = d_s = 1$. Then, $f_1,\dots,f_s$ form a homogeneous linear system in $n$ variables:\[
    \begin{split}
    f_1(\xx) &= c_{1,1} x_1 + c_{1,2} x_2 + \dots + c_{1,n} x_n\\
    &\;\;\vdots\\
    f_s(\xx) &= c_{s,1} x_1 + c_{s,2} x_2 + \dots + c_{s,n} x_n
    \end{split}
    \]
    In this case, the resultant variety coincides with the determinantal variety of $s\times n$ matrices of rank at most $n-1$.
    That is, the system has a non-zero solution if and only if the matrix $C\coloneqq [c_{ij}]_{i\in [s], j \in [n]}$ has $\rk C<n$.
    
    Consequently, the resultant variety is defined by the $\binom{s}{n}$ maximal minors of $C$. 
    Intriguingly, there is a substantially smaller resultant system:
    Bruns and Schw\"{a}nzl \cite{Bruns-89,Bruns-Schwanzl-90} proved that a collection of $N\coloneqq (s-n)n+1$ polynomials is necessary and sufficient.
    Namely, there exists a collection $R_1,\dots,R_{N}$ of polynomials in variables $c_{i,j}, i\in [s], j\in [n]$ such that $\rk C<n$ if and only if $R_\ell(C)=0$,
    for all $\ell=1,\dots,N$.
\end{example}

\begin{example}
\label{ex:macaulay}
A classical result due to Macaulay \cite{Macaulay-02,Macaulay-16} states that if \eqref{eq:system} is a square system, i.e., if $s=n$, then the resultant variety is a hypersurface whose defining equation is \emph{Macaulay's resultant} $\Res(f_1,f_2,\dots,f_n)$: \[
 \ff = 0 \text{ has a non-zero solution}\quad\iff\quad\Res(f_1,f_2,\dots,f_{n}) = 0.
\] 
Furthermore, in this case, the resultant system consists of a single multihomogeneous polynomial, whose degree in the coefficients of $f_i$ is $\prod_{j\neq i} d_j$, see e.g. \cite[Theorem~(3.1)]{Cox-Little-Oshea-05}.
If $n=s=2$, i.e., in the case of two homogeneous bivariate polynomials $f_1 = \sum_{\ell=0}^d a_{\ell} x^\ell y^{d-\ell}$ and $f_2 = \sum_{\ell=0}^e b_{\ell} x^{\ell} y^{e-\ell}$, 
Macaulay's resultant has the following simple form as the determinant of the $(d+e)\times (d+e)$ \emph{Sylvester matrix} whose entries are the coefficients of $f_1,f_2$: \[
\Res(f_1,f_2) \;\coloneqq \;\begin{vmatrix}
    a_0 & a_1 & \dots & a_d & 0 & \dots & 0\\
    0 & a_0 & a_1 & \dots  & a_d & \dots & 0\\
    \vdots & \vdots & \ddots & \ddots & \ddots & \ddots & \vdots\\
    0 & \dots & 0 & a_0 & a_1 & \dots & a_d\\
    b_0 & b_1 & \dots & b_e & 0 & \dots & 0\\
    0 & b_0 & b_1 & \dots  & b_e & \dots & 0\\
    \vdots & \vdots & \ddots & \ddots & \ddots & \ddots & \vdots\\
    0 & \dots & 0 & b_0 & b_1 & \dots & b_e
\end{vmatrix}
\]
This polynomial $\Res(f_1,f_2)$ is the \emph{Sylvester resultant} of two forms in two variables \cite{Sylvester-40}.
It is a bihomogeneous polynomial in the coefficients of $f_1,f_2$ of multidegree $(e,d)$.
\end{example}

We give a brief overview of our contributions so that it is easier for the reader to compare our results with the existing and related works presented in Section~\ref{sec:prev-works}.
We elaborate in more detail on our contribution in Section~\ref{sec:main-results}.
 
If $s>n$, the resultant variety is no longer a hypersurface, so we cannot  describe it as the zero locus of a single polynomial.
We construct resultant systems significantly smaller than previously known.
Our main results, informally, are:
\begin{itemize}
    \item For $s$ homogeneous polynomials of uniform degree $d$, we prove the existence of a resultant system of cardinality $\binom{d+n-1}{n-1}s-n^2+1$ (\cref{thm:main-optimal}).
    \item For the case of linear systems, that is $d=1$, we present a new description of a set of $(s-n)n+1$ polynomials that define the determinantal variety of rank deficient matrices (\cref{cor:linear-resultant-system}).
    \item For bivariate homogeneous polynomials ($n=2$) we demonstrate that our constructions give resultant systems with cardinality that matches, up to a multiplicative constant, a known  lower bound by Lyubeznik~\cite{Lyubeznik-95}.
    \item For the general case, we construct an explicit \emph{punctured resultant system}, i.e., a resultant system under the assumption that $f_\ell\neq 0$, for all $\ell=1,2,\dots,s$.
    It contains $\poly(s,d)$ polynomials, when $n$ is fixed (\cref{thm:intro-explicit-multivariate}).
    For bivariate homogeneous systems, this construction consists of $(s-2)d+1$ polynomials~(\cref{thm:bivariate-explicit}).
\end{itemize}

\subsection{Previous and related work}
\label{sec:prev-works}
We now summarize the previous work on resultant systems. 

\subsubsection{Bivariate systems}

First, we focus on systems of forms in two variables, that is:
\begin{equation}
	\label{eq:biv-system}
	\ff(\xx) = 0 \Leftrightarrow
\begin{cases}
    f_1(x_1, x_2) &\coloneqq \sum_{\ell=0}^{d_1} c_{1,\ell} x_1^{d_1-\ell}x_2^\ell =\,0 \\
    &\;\vdots  \\
    f_s(x_1, x_2) &\coloneqq \sum_{\ell=0}^{d_s} c_{s,\ell} x_1^{d_s-\ell}x_2^\ell =\,0
\end{cases} 
\enspace .
\end{equation}

The problem of constructing a resultant system for bivariate homogeneous systems was considered by Kapferer \cite{Kapferer-29} (see also van der Waerden \cite{vdWaerden-40} and Hodge-Pedoe \cite{Hodge-Pedoe-94}).
He showed that one can always reduce to the equal--degree case $d_1=d_2=\dots=d_s=d$, by replacing each polynomial $f_i, i=2,\dots,s$ with the pair $x_1^{d_1-d_i}f_i, x_2^{d_1-d_i}f_i$.
A moment's thought shows that this does not change the solution set of the original system.

Then, he introduces $2s$ new variables, namely $\lambda_1,\dots,\lambda_s$ and $\mu_1,\dots,\mu_s$, and  expands the Sylvester resultant (\cref{ex:macaulay}) of two generic linear combinations $\sum_i\lambda_i f_i, \sum_i \mu_i f_i$ in the variables $\lambda_i,\mu_i$: \[
\Res\left( \lambda_1 f_1 + \dots + \lambda_s f_s,\; \mu_1 f_1 + \dots + \mu_s f_s \right) = \sum_{|\alpha|=|\beta|=d} R_{\alpha,\beta}(c_{i,j}) \, \lambda_1^{\alpha_1}\dots\lambda_s^{\alpha_s} \, \mu_1^{\beta_1}\dots \mu_s^{\beta_s}.
\] Kapferer proved that the coefficients $R_{\alpha,\beta}, |\alpha|=|\beta|=d$, form a resultant system.
This approach yields a resultant system of cardinality $\binom{d+s-1}{s-1}^2$, since this is the number of coefficients of a generic bihomogeneous polynomial of degree $(d,d)$ in $2s$ variables.\footnote{We remark that some repetitions and cancellations can occur. For example, when $d$ is odd we have for every $\alpha$ and $\beta$, $R_{\alpha,\alpha}$ is identically zero; whereas $R_{\alpha,\beta}=- R_{\beta,\alpha}$. 
We are not aware of a general cancellation rule that reduces the number of polynomials in a substantial way.}
We note that this number is exponential in $d$ and/or $s$.

An alternative method is given in \cite{Orsinger-51,Kakie-76}.
In this construction, one obtains a resultant system as the set of maximal minors of a \emph{generalized Sylvester matrix}.
We denote this matrix by $S_r(\ff)$, 
where $r$ is a natural number~$r\geq d_1$. 
It has dimension $\left(s(r+1)-\sum_{i=1}^s d_i\right)\times (r+1)$
and is defined as follows:
\begin{equation}
	S_r(\ff)\coloneqq \begin{pmatrix}
		c_{1,0} & c_{1,1} & \dots & c_{1,d_1} & &  \\
		 & \ddots & \ddots & & \ddots & \\
		& & c_{1,0} & c_{1,1} & \dots & c_{1,d_1} \\
		\vdots & \vdots & \vdots & \vdots & \vdots & \vdots\\
		c_{s,0} & c_{s,1} & \dots & c_{s,d_s} & & \\
		 & \ddots & \ddots & & \ddots & \\
		 & & c_{s,0} & c_{s,1} & \dots & c_{s,d_s} 
	\end{pmatrix} 
	\enspace,
\end{equation}
where the coefficients of each $f_i$ are repeated in $r-d_i+1$ rows and the blank spaces are filled with zeros.
If $s=2$ and $r=d_1+d_2-1$, then $S_r(\ff)$ is simply the Sylvester matrix of $f_1$ and $f_2$.
Orsinger \cite{Orsinger-51} proved that the maximal minors of $S_r(\ff)$ form a resultant system for $r = 2d_1 -1$.
Kaki\'{e} \cite{Kakie-76} proved that $r=d_1+d_s-1$ suffices.
In particular, for any admissible value of $r$ in these results, 
the polynomials in \eqref{eq:biv-system} have a common non-zero solution if and only if the generalized Sylvester matrix is such that $\rk(S_r(\ff))<r+1$.
In the most recent addition in this line of work, 
Conner et al.~\cite{CMSS-25}  proved, in the equal--degree case, that the maximal minors of $S_r(\ff)$ for $r=d,\dots,2d-1$ define $\cR_{d,\dots,d}$ \emph{ideal theoretically}, i.e., they generate the ideal of the resultant variety.

We assume for simplicity that $d_1=\dots=d_s=d$ and we analyze the cardinality of Orsinger and Kaki\'{e}'s resultant system.
In this case, we have $r=2d-1$ and the generalized Sylvester matrix $S_{2d-1}(\ff)$ has $\binom{sd}{2d}$ maximal minors. 
For sufficiently large $s$ and $d$, this number is substantially larger than the number of polynomials obtained by Kapferer's construction.

However, it is possible to reduce the cardinality using the results of \cite{Bruns-89,Bruns-Schwanzl-90}:\footnote{We have not encountered this observation in the literature.} 
Recall from \cref{ex:maximal-minors} that there are $(sd)(2d)-(2d)^2+1=(2s-4)d^2+1$ polynomials that cut out the determinantal variety of $(sd)\times(2d)$ matrices of rank $<2d$.
An application of this result to $S_{2d-1}(\ff)$ yields a resultant system of cardinality $\cO(s d^2)$.

Lyubeznik \cite{Lyubeznik-95} investigated lower bounds for resultant systems of bivariate polynomials.
Assuming $d_1\geq d_2\geq\dots\geq d_s$, he proved a lower bound \begin{equation}
\label{eq:lyubeznik-lower-bound}
N_s \coloneqq 1 + \sum\nolimits_{i=3}^s (d_i + 1)
\end{equation} when $d_1 \neq d_3$. 
Furthermore, he proved that a resultant system with $N_s$ polynomials exists if $s=3$, or $d_2=\dots=d_s=d_1-1$ and each $f_i$ is assumed to be non-zero.
Lyubeznik's construction of resultant systems consists of polynomials of the form \begin{equation}
\label{eq:lyubeznik-system}
\Res(\,\lambda^{(i)}_1 f_1+\dots+\lambda^{(i)}_s f_s,\; \mu^{(i)}_1 f_1 + \dots + \mu^{(i)}_s f_s\,)
\end{equation} for an explicit collection of coefficients $(\lambda^{(i)},\mu^{(i)}), i=1,2,\dots, N_s$.
Lyubeznik asked if a resultant system of cardinality $N_s$ always exists \cite[Questions 1 and 2]{Lyubeznik-95}.
By building upon Lyubeznik's construction, Encarnaci\'{o}n \cite{Encarnacion-98} proved that an explicit resultant system of the form \eqref{eq:lyubeznik-system} with $\cO(s^2d)$ polynomials exists.

We observe that there is a gap between the lower bound and the upper bounds: 
Encarnaci\'{o}n's \cite{Encarnacion-98} construction and Orsinger-Kaki\'{e} \cite{Orsinger-51,Kakie-76} construction via maximal minors of $S_{2d-1}(\ff)$ reinforced by the Bruns-Schw{\"a}nzl reduction in the number of maximal minors yield resultant systems of cardinalities $\cO(s^2d)$ and $\cO(sd^2)$, respectively.
However, Lyubeznik's \cite{Lyubeznik-95} lower bound is $\Omega(sd)$.

In \cref{cor:bivariate,thm:bivariate-explicit} below, we close this gap by demonstrating the existence of resultant systems for $n=2$ of cardinality $\cO(sd)$.
These resultant systems also have the form \eqref{eq:lyubeznik-system}, for a collection of appropriately chosen $\lambda^{(i)},\mu^{(i)}$.

\subsubsection{Systems with more than two variables}
For general $n\geq 2$, Kapferer \cite{Kapferer-29} constructed a resultant system based on an inductive procedure that eliminates one variable at each step.\footnote{According to van der Waerden \cite{vdWaerden-40} and Orsinger \cite{Orsinger-51}, this method is known as \emph{Kronecker's method of elimination}. }
Following the presentation of Lang~\cite{Lang-12}, 
this construction introduces a new variable $t$ and constructs a new system 
\[
    \mathbf{g}(\xx, t) = 0 \enspace, \qquad
    \text{ where }\,
    g_i(x_1,\dots,x_n,t) \coloneqq f_i(t x_1, \dots, t x_{n-1}, x_n), \quad \, \text{for }\, i =1,2,\dots,s.
\] A moment's thought shows that \eqref{eq:system} has a non-zero solution if and only if $g_1,\dots,g_s$ has a solution $(x_1,\dots,x_n,t)$ with $(x_1,\dots,x_{n-1})\neq 0$ and $(x_n,t)\neq 0$.
This suggests the following algorithm to compute a resultant system:
View $g_1,\dots,g_s$ as a system of bivariate forms in the variables $(x_n,t)$ and coefficients in $\KK[x_1,\dots,x_{n-1}]$, then eliminate $x_n$ and $t$ by using a resultant system for bivariate polynomials.
The resulting system consists of polynomials in the variables $x_1,\dots,x_{n-1}$, which has a solution with $(x_1,\dots,x_{n-1})\neq 0$ if and only if the original system has a non-zero solution.
By eliminating one variable at each step, one can construct a resultant system for forms in $n$ variables.
Unfortunately, at each step, the degree of the polynomials is squared.
Thus, even if we choose the resultant system for bivariate systems of the smallest cardinality, this method yields a resultant system that consists of a doubly exponential, in $n$, number of polynomials, with degrees up to $\cO(d^{2^{n-1}})$, when $n$ is fixed. 

Yap \cite{Yap-99} presents a different construction,
based on the maximal minors of a coefficient matrix.
Let $D_1,\dots,D_s,D$ be positive integers, 
such that $\forall i, D_i + d_i =D$.
We consider the map\footnote{When $s=n$ and $D=d_1+\dots+d_n+n-1$, this matrix is known as the \emph{Macaulay matrix}. 
In this case, the greatest common divisor of the maximal minors of $M$ equals Macaulay's resultant, see \cite[\S 1]{Macaulay-16}.} \[
M : \prod_{\ell=1}^s \KK[\xx]_{D_{\ell}} \rightarrow \KK[\xx]_D, \quad (q_1,\dots,q_s) \mapsto \sum_{\ell=1}^s f_\ell q_\ell \enspace.
\] For $n=2$, the matrix of $M$ in the standard monomial basis coincides with the transpose of the generalized Sylvester matrix $S_D(\ff)$.
For general $n$ in the equal--degree case $d_1=\dots=d_s=d$, $M$ is a $\binom{D-d+n-1}{n-1}s$ by $\binom{D+n-1}{n-1}$ matrix, whose entries are the coefficients $c_{i,\alpha}$.
Yap proved that the maximal minors of $M$ form a resultant system whenever $D\geq 1+n(B-1)$ where $B=\cO(d^n)$ is the \emph{effective Nullstellensatz bound}, see \cite[Lecture XI, Theorem 13]{Yap-99}. 
When $n$ is fixed, this yields a resultant system with $s^{\cO(d^{n(n-1)})}$ equations of degrees $\cO(d^{n(n-1)})$.

Abramov \cite{Abramov-13} generalized Kapferer's method of obtaining a resultant system as the coefficients of a resultant of generic linear combinations.
As before, we will assume without loss of generality that $d_1=\dots=d_s=d$, by replacing each $f_i, i=2,\dots,s$ with $n$ polynomials $x_1^{d_1-d_i}f_i, x_2^{d_1-d_i}f_i,\dots, x_n^{d_1-d_i}f_i$.  

We consider the $n\times s$ matrix $\Lambda$ whose entries are new variables $\Lambda_{ij}, i\in [n], j\in [s]$. 
Abramov proved that the coefficients of the expansion of Macaulay's resultant of $n$ linear combinations of $f_1,f_2,\dots,f_s$ in the variables $\Lambda_{ij}$ form a resultant system:
\begin{equation}
\label{eq:abramov-resultant}
    \Res\left(\, \Lambda_{11}f_1 + \dots + \Lambda_{1s}f_s ,\, \dots ,\, \Lambda_{n1}f_1 + \dots + \Lambda_{ns} f_s \,\right) = \sum_{\forall \ell, |\alpha_\ell|=d^{n-1}} R_{\alpha_1,\dots,\alpha_n} \, \Lambda_{1,\cdot}^{\alpha_1}\, \Lambda_{2,\cdot}^{\alpha_2}\,\dots\,\Lambda_{n,\cdot}^{\alpha_n},
\end{equation} 
where $\Lambda_{i,\cdot}$ denotes the collection of the variables $\Lambda_{ij}, j=1,\dots,s$ in the $i$-th row of $\Lambda$.

\begin{restatable}[{\cite{Abramov-13}}]{theorem}{abramov}
\label{thm:intro-abramov}
    The polynomials $R_{\alpha_1,\dots,\alpha_n}$ in \eqref{eq:abramov-resultant} form a resultant system for $s$ homogeneous forms of degree $d$ in $n$ variables.
\end{restatable}
\noindent
In \cref{sec:prelim-system}, we will give a simple proof of \cref{thm:intro-abramov} using Perron's theorem \cite{Perron-41}.

We observe that the number of $R_{\alpha_1,\dots,\alpha_n}$'s coincides with the number of coefficients of a generic multihomogeneous form of multidegree $(d^{n-1},\dots,d^{n-1})$ in $sn$ variables.
Hence, Abramov's resultant system contains $\binom{d^{n-1}+s-1}{s-1}^{n}$ polynomials.
This number is exponential in $d$ and/or $s$ even when $n$ is fixed.
Our aim is to reduce this number by considering resultants of the form \eqref{eq:abramov-resultant} for specific evaluations at matrices $\Lambda_1,\dots,\Lambda_N\in\KK^{n\times s}$, for a carefully chosen $N$.

\subsection{Main results}
\label{sec:main-results}

We present in detail our contributions. 
We make the assumption that $d_1=\dots=d_s=d$.
Recall that this is a harmless assumption since we can always replace the polynomials $f_i, i=2,\dots,s$ with $x_1^{d-d_i}f_i,\dots,x_n^{d-d_i}f_i$.
This increases the number of polynomials by at most $(s-1)(n-1)$.

Our first and  main theorem  supports the existence of a resultant system with cardinality $\binom{d+n-1}{n-1}s-n^2+1$. 
We remark that this number is $\Theta(s d^{n-1})$ for constant $n$.
To the best of our knowledge, this is the first resultant system in the literature that is not exponentially large in $s$ or $d$ for $n>2$.
\begin{restatable}{theorem}{mainthm}
\label{thm:main-optimal}
	Suppose $s>n$ and $d_1=\dots=d_s=d$.
	If $\cH\subset\KK^{n\times s}$ is a generic set of $n\times s$-matrices of cardinality $|\cH|\geq \binom{d+n-1}{n-1}s-n^2+1$, then the following collection of polynomials 
	\begin{equation}
    \label{eq:resultant-system}
	\big\{\,\Res(\Lambda_{11}f_1+\dots+\Lambda_{1s}f_s\,,\, \dots,\, \Lambda_{n1}f_1+\dots+\Lambda_{ns}f_s)\;\mid\; \Lambda\in\cH\,\big\}
	\end{equation}  
	is a resultant system for $\ff = 0$ \eqref{eq:system}.
\end{restatable}

The cardinality $\binom{d+n-1}{n-1}s-n^2+1$ is not a coincidence:
it equals the dimension of the projective GIT quotient of the space of $s$-tuples of forms of degree $d$ by the natural $\SL(n)$-action via change of variables (plus one); 
see \cref{sec:git} for further details.

If $d=1$, then the cardinality of the resultant system 
of \cref{thm:main-optimal} coincides with the number of polynomials that define the determinantal variety of rank deficient $s\times n$-matrices (\cref{ex:maximal-minors}).
In this setting, an application of \cref{thm:main-optimal} is as follows:
\begin{restatable}{corollary}{linearcor}
\label{cor:linear-resultant-system}
    Suppose that $s>n$ and $\Lambda_1,\Lambda_2,\dots,\Lambda_{N}\in\KK^{n\times s}$ is a generic collection of $N\coloneqq sn-n^2+1$ matrices.
    Then, for every $C\in\KK^{s\times n}$ we have \[
    \rk C < n \quad\iff\quad \forall i\in [N], \; \det(\Lambda_i C)=0.
    \]
\end{restatable}
To the best of our knowledge, this is a novel description of $N$ polynomials that define the determinantal variety of matrices of rank $<n$.
We also remark that it differs from the results of Bruns and Schw{\"{a}}nzl~\cite{Bruns-Schwanzl-90},
who proved that certain sums of the form \[
q_r \coloneqq \sum_{\substack{1\leq j_1 < j_2 < \dots < j_n \leq s\\ \sum_{\ell=1}^n j_\ell-\ell=r-1}} [j_1, j_2,\dots, j_n],\qquad r=1,2,\dots,N,
\] where $[j_1,\dots,j_n]$ denotes the maximal minor corresponding to the rows $j_1,\dots,j_n$, define the determinantal variety of $s\times n$ matrices of rank $<n$, see also \cite{Nie-10}.

We are not aware of any method of computing the sums $q_r, r=1,2,\dots,N$ without actually computing all of the $\binom{s}{n}$ maximal minors $[j_1,\dots,j_n]$.
By comparison, if the (generic) collection $\Lambda_1,\dots,\Lambda_N$ with the property in \cref{cor:linear-resultant-system} is known in advance, then the polynomials $\det(\Lambda_i C)$ are easy to evaluate. 
We leave the construction of an explicit tuple $(\Lambda_1,\dots,\Lambda_N)$ with this property as an open problem.

Regarding the comparison of \cref{thm:main-optimal} 
to related results for the bivariate case, we have the following immediate corollary.
\begin{restatable}{corollary}{bivariatecor}
\label{cor:bivariate}
    Suppose $n=2$ and $d_1=\dots=d_s=d$. Then there is a resultant system that consists of $(d+1)s-3$ polynomials.
\end{restatable}
\noindent
Recall that the best known upper bounds for bivariate systems were $\cO(s^2d)$ by Encarnaci\'{o}n \cite{Encarnacion-98} and $\cO(sd^2)$ by Orsinger \cite{Orsinger-51} and Kaki\'{e} \cite{Kakie-76} after applying the Bruns-Schw\"{a}nzl reduction in the number of maximal minors needed to define the determinantal variety.
In comparison, \cref{cor:bivariate} yields a resultant system with $(d+1)s-3$ polynomials in the equal--degree case, and $(2s-1)(d+1)-3$ in general by replacing each $f_i$ with $x_1^{d_1-d_i}f_i, x_2^{d_1-d_i}f_i$, matching Lyubeznik's lower bound \eqref{eq:lyubeznik-lower-bound} up to a factor of $2$.

\subsubsection{Punctured resultant systems}
We now turn to a related question:
What is the cardinality of the smallest resultant system if we assume $f_\ell\neq 0$ for every $\ell$;
for example by assuming that all the leading coefficients  are non-zero?
We will call a collection of polynomials $R_1,R_2,\dots,R_{N}$ a \emph{punctured resultant system} for forms of degrees $d_1,d_2,\dots,d_s$ if for every $\ff\in\cP^{d_1,\dots,d_s}$ with $\forall \ell=1,\dots,s, f_\ell\neq 0$ we have \[
\cZ_{\PP}(f_1,\dots,f_s)\neq\varnothing\quad\iff\quad R_1(\ff)=\dots=R_N(\ff)=0.
\]

We prove in the bivariate case that there is an explicit choice of $(s-2)d+1$ linear combinations whose resultants form a punctured resultant system.
Note that this number is smaller than $(d+1)s-3$ that we stated in \cref{cor:bivariate}.
Moreover, the choice of the linear combinations is completely explicit and does not rely on the choice of a generic point. 
\begin{restatable}{theorem}{puncturedbivariate}
\label{thm:bivariate-explicit}
	Suppose $n=2$ and $d_1=\dots=d_s=d$. 
	Then, there is an explicit punctured resultant system of cardinality $(s-2)d+1$.
\end{restatable}
\noindent
The explicit system that we construct consists of a collection of Sylvester resultants of a linear combination of the first $s-1$ polynomials against $f_s$, i.e., \[
R_i\coloneqq \Res\left(\, f_1 + \lambda_i f_2 + \dots + \lambda_i^{s-2} f_{s-1} ,\; f_s\, \right), \quad i=1,\dots,(s-2)d+1,
\] 
where, in this case, $\Res$ denotes the Sylvester resultant and $\lambda_i\in\KK, i=1,\dots,(s-2)d+1$ are arbitrary distinct scalars, e.g., we may choose them to be the first $(s-2)d+1$ natural numbers.
We remark that we do not assume genericity; any collection of distinct scalars is sufficient.

For systems with $n>2$, we can again construct an explicit punctured resultant system.
Unfortunately, the cardinality of the resultant system that we construct is no longer comparable to  \cref{thm:main-optimal};
but it is $\poly(s,d)$ if $n$ is constant.

\begin{theorem}[\cref{thm:explicit-multivariate}--informal]
    \label{thm:intro-explicit-multivariate}
    Suppose $n>2$ is fixed.
    Then there is an explicit punctured resultant system of cardinality $ \cO(s^{n-1} d^{n(n-1)/2})$.
\end{theorem}
\noindent
As in the bivariate case, the punctured resultant system consists of polynomials of the form \[
R_i\coloneqq \Res\left(\, \sum_{\ell=1}^{s-1} \lambda_i^{\ell-1} f_\ell, \, \sum_{\ell=1}^{s-1} \mu_i^{\ell-1} f_\ell,\, \dots,\, \sum_{\ell=1}^{s-1} \gamma_i^{\ell-1} f_\ell, \, f_s\,\right),\quad i=1,\dots,\cO(s^{n-1}d^{n(n-1)/2}),
\] where $\Res$ denotes Macaulay's resultant and $\lambda_i,\mu_i,\dots,\gamma_i\in\KK$, for $i=1,\dots,\cO(s^{n-1}d^{n(n-1)/2})$, are distinct scalars.

\subsection{Outline of the paper}
In the preliminary \cref{sec:prelims}, we collect known facts from algebraic geometry.
We prove the existence of the resultant variety and its dimension formula and provide an elementary proof of \cref{thm:intro-abramov} using Perron's theorem.
\cref{sec:main} is dedicated to the proof of \cref{thm:main-optimal}.
We start by proving \cref{cor:linear-resultant-system} in \cref{sec:linear-case} as a preparation for the proof of \cref{thm:main-optimal}.
The main technicality in the proof of \cref{thm:main-optimal} is the construction of the \emph{projective GIT quotient} of the space of tuples of polynomials in $n$ variables by the action of $\SL(n)$.
We give this construction in \cref{sec:git} and give a complete proof of \cref{thm:main-optimal} in \cref{sec:proof-main}.
In \cref{sec:multiproj}, we prove \cref{thm:bivariate-explicit,thm:intro-explicit-multivariate} by relying on pigeonhole--type arguments in intersection theory.

\section{Preliminaries}
\label{sec:prelims}
We briefly review some facts from algebraic geometry.
For a detailed discussion of the various topics, we refer to \cite{Harris-93,Shafarevich-13}. 
\subsection{Affine and projective varieties}
\label{sec:varieties}

We fix an algebraically closed field $\KK$ of characteristic $0$ and let
$\KK[x_1,\dots,x_n]$ denote the ring of polynomials in $n$ variables with coefficients in $\KK$. 
We will also use the notation $\KK[\xx]$ if the number of variables is clear from the context.
Given a natural number $d$, $\KK[\xx]_d$ denotes the vector space of homogeneous polynomials of degree $d$.

We denote by $\KK^n$ the $n$-dimensional vector space over $\KK$.
A polynomial $f\in\KK[x_1,\dots,x_n]$ defines a function on $\KK^n$, which maps $(a_1,\dots,a_n)$ to $f(a_1,\dots,a_n)$.
If $f_1,\dots,f_s\in\KK[\xx]$, then we denote 
their zero locus in $\KK^n$ by
  \[
\cZ(f_1,\dots,f_s) \coloneqq \{x\in\KK^n\mid f_1(x)=\dots=f_s(x)=0\} .
\] 
We call a set of this form an \emph{affine variety}.
We have $\cZ(f_1,\dots,f_s)=\cZ(I)$ where $I=(f_1,\dots,f_s)$ is the ideal of $\KK[\xx]$ generated by $f_1,\dots,f_s$.

We obtain the \emph{Zariski topology} on $\KK^n$ by declaring affine varieties as closed sets.
An affine variety $X\subset\KK^n$ is called \emph{irreducible} if it cannot be written as the union of two proper closed subsets. 
Every affine variety admits a decomposition as a finite union $X=X_1\cup\dots\cup X_c$  of irreducible varieties.

For a subset $X\subset\KK^n$, we define \[
I(X) \coloneqq \{ f\in \KK[\xx] \mid \forall x\in X, f(x)=0 \}
\] to be the \emph{ideal} of $X$. 
Hilbert's Nullstellensatz (see e.g. \cite[Chapter 4, \S 1]{Cox-Little-Oshea-25}) states that \[
I(\cZ(f_1,\dots,f_s)) = \sqrt{(f_1,\dots,f_s)},
\] where $\sqrt{\cdot}$ denotes the \emph{radical}.
If $X\subset\KK^n$ is an affine variety, then we denote by \[
\KK[X] \coloneqq \KK[\xx]/I(X)
\] its \emph{affine coordinate ring}.
Note that $\KK[\KK^n]\cong\KK[x_1,\dots,x_n]$.
There is a one-to-one correspondence between points $p\in X$ and maximal ideals $\mathfrak{m}\subset\KK[X]$ via $p\mapsto I(\{p\})=(x_1-p_1,\dots,x_n-p_n)$.
In particular, we have $p=p'$ if and only if $I(\{ p\})=I(\{p'\})$.

Conversely, if $A$ is a finitely generated $\KK$-algebra, then
$A \cong \KK[x_1,\dots,x_n]/I$, for some ideal
$I=\ideal{f_1,\dots,f_s}\subset \KK[x_1,\dots,x_n]$.
The closed points of $\Spec A$ are identified with the affine variety
$\cZ(I)\subset \KK^n
$.
If $A$ is reduced, equivalently if $I$ is radical, then $A$ is the coordinate
ring of the affine variety $\cZ(I)$.

There is an alternative definition that does not rely on choosing generators. 
We define $\Spec A$ to be the set of prime ideals of $A$.
We define the Zariski topology on $\Spec A$ whose closed sets are of the form $\cZ(I)\coloneqq\{\mathfrak{p}\in \Spec A\mid I\subset \mathfrak{p}\}$ where $I$ denotes an ideal of $A$.

We denote by $\PP^{n-1}$ the $(n-1)$-dimensional projective space over $\KK$; defined as the set of one--dimensional linear subspaces of $\KK^n$.
If $f_1,\dots,f_s$ are homogeneous polynomials, then we denote by \[
\cZ_{\PP}(f_1,\dots,f_s) \coloneqq \{[x]\in\PP^{n-1}\mid f_1(x)=\dots=f_s(x)=0\}
\] their zero locus in $\PP^{n-1}$.
Here, $[x]$ denotes the line in $\KK^n$ that goes through $x\neq 0$.
We call a set of the form $\cZ_{\PP}(f_1,\dots,f_s)$ a \emph{projective variety}.
If $\cZ_{\PP}(f_1,\dots,f_s)\neq\varnothing$, we say that $f_1,\dots,f_s$ have a common non-zero solution.
We note that the Cartesian product of two projective varieties is also a projective variety via the \emph{Segre embedding}, see \cite[Example~2.21]{Harris-93}.

The Zariski topology on $\PP^{n-1}$ is defined by defining the projective varieties to be closed sets.
If $X\subset\PP^{n-1}$ is a projective variety, then we define the \emph{affine cone} over $X$ as \[
C(X) \coloneqq \{x\in \KK^n\mid [x]\in X\} \cup \{0\}.
\] Note that $C(X)=\cZ(f_1,\dots,f_s)$ is an affine variety.
The \emph{homogeneous coordinate ring} of a projective variety $X$ is defined as the affine coordinate ring $\KK[C(X)]$ of $C(X)$.
We note that $\KK[C(X)]$ admits a natural grading $\KK[C(X)]=\oplus_{d=0}^\infty \KK[\xx]_d/I(C(X))_d$.
This is well-defined since $I(C(X))$ is a homogeneous ideal.
There is a correspondence between points $[x]\in X$ and homogeneous prime ideals of $\KK[X]$, given by $[x]\mapsto \mathfrak{p}_{x}\coloneqq 
\ideal{f\in \KK[X]\mid f\text{ is homogeneous and } f(x)=0}$.
We have $[x]=[x']$ if and only if $\mathfrak{p}_x=\mathfrak{p}_{x'}$.

If $A=\oplus_{d=0}^{\infty} A_d$ is a \emph{graded ring}, then there is a corresponding projective variety $\Proj A$ defined as follows:
If $A$ is generated by homogeneous elements $f_1,\dots,f_s$ of the same degree, then we have a surjection \[
\KK[y_1,\dots,y_s]\rightarrow A, \quad y_i\mapsto f_i.
\] The kernel $I$ of this map is a homogeneous ideal and we define $\Proj A\coloneqq \cZ_{\PP}(I)\subset\PP^{s-1}$.
If $A$ cannot be generated by elements of the same degree, then we realize $\Proj A$ as an algebraic subset of the \emph{weighted projective space}, which is itself a projective variety.

An alternative description of $\Proj A$ operates as follows.
We call $A_{+}\coloneqq\oplus_{d=1}^\infty A_d$ the \emph{irrelevant ideal} of $A$. 
We define $\Proj A$ to be the set of homogeneous prime ideals of $A$ that do not contain $A_{+}$:\[
\Proj A\coloneqq \{ \mathfrak{p}\mid \mathfrak{p}\text{ is a homogeneous prime ideal of }A \text{ and } A_{+}\not\subset \mathfrak{p} \}.
\] We define the Zariski topology on $\Proj A$ by defining the closed sets to be those of the form $\cZ(I)=\{\mathfrak{p}\in\Proj A\mid I\subset \mathfrak{p}\}$ for a homogeneous ideal $I$.
This defines the structure of a projective variety on $\Proj A$.

An open subset $X$ of a projective variety is called a \emph{quasi-projective variety}. 
The dimension of a quasi-projective variety is defined as the dimension of its closure $\dim X\coloneqq \dim\overline{X}$.
We will use the following theorem for dimension computations.

\begin{theorem}[{\cite[Theorem~11.12]{Harris-93}}]
\label{thm:harris-dim}
    Let $X$ be a quasi-projective variety and $\phi:X\rightarrow\PP^n$ be a map.
    Let $Y$ be the Zariski closure of $\phi(X)$.
    For $p\in X$, let $X_p=\phi^{-1}(\phi(p))$ be the fiber of $\phi$ through $p$, and let $\mu(p)\coloneqq \dim_p(X_p)$ be the local dimension of $X_p$ at $p$.
    Then, $\mu(p)$ is an upper semi-continuous function with respect to the Zariski topology on $X$, i.e., its superlevel sets $\{p\in X\mid \mu(p)\geq r\}$ are Zariski closed.
    Moreover, if $X_0$ is an irreducible component of $X$, $r$ is the minimum value of $\mu(p)$ on $X_0$ and $Y_0\subset Y$ is the closure of the image of $X_0$, then \[
    \dim X_0 = \dim Y_0 + r.
    \]
\end{theorem}

As a corollary we obtain the following:
\begin{corollary}
\label{cor:quasi-proj-dim}
   If $X$ is a quasi-projective variety, $\phi:X\rightarrow\PP^{n}$ is a map and all non-empty fibers have the same global dimension $r$, then \[
   \dim X = \dim\overline{\phi(X)} + r.
   \]
\end{corollary}
\begin{proof}
Set $Y=\overline{\phi(X)}$.
    Suppose $X=X_1\cup \dots\cup X_c$ is the irreducible decomposition of $X$ and define $Y_\ell\coloneqq\overline{\phi(X_\ell)}$.
    We have $\dim X =\max_{1\leq \ell\leq c}\dim X_\ell $.
    For $p\in X_{\ell}$, we have \[
    \dim_p(\phi^{-1}(\phi(p))\cap X_{\ell}) \leq \dim\phi^{-1}(\phi(p)) = r,
    \] where $\dim_p$ denotes the local dimension at $p$.
    Hence, \cref{thm:harris-dim} gives \[
    \dim(X_{\ell}) = \dim Y_\ell + \min_{p\in X_{\ell}}\dim_p\left(\phi^{-1}(\phi(p))\cap X_{\ell}\right) \leq \dim Y+r.
    \] Since $\ell$ was arbitrary, we deduce $\dim X\leq \dim Y +r$.
    
    For the opposite inequality, let $Z$ be a component of $Y$ of maximal dimension.
    Since $Y$ is covered by the irreducible varieties $\overline{\phi(X_\ell)}, \ell=1,\dots,c$, there exists $\ell$ such that $\overline{\phi(X_\ell)}=Z$.
Among all such components, at least one of them has the property that $\mu(p)$ takes the value $r$ on $X_\ell$.     
By applying \cref{thm:harris-dim} we obtain $\dim X_{\ell}=\dim Z + r=\dim Y+r$.
    This finishes the proof.
\end{proof}

We recall that if $X$ is a projective variety and $\phi:X\rightarrow\PP^{n}$ is a map, then $\phi(X)$ is Zariski closed, see \cite[Theorem~3.12]{Harris-93}. 
In this case, \cref{thm:harris-dim} implies the following:
\begin{lemma}[{\cite[Corollary~11.13]{Harris-93}}]
\label{lem:dimension}
    Suppose $X$ is a projective variety, $\phi:X\rightarrow\PP^{n}$ is a map and $Y=\phi(X)$.
    For $q\in Y$ define $\lambda(q)\coloneqq \dim\phi^{-1}(q)$.
    Then, $\lambda(q)$ is an upper semi-continuous function of $q$ with respect to the Zariski topology on $Y$, i.e., its superlevel sets $\{q\in Y\mid \lambda(q)\geq r\}$ are Zariski closed.
    Moreover, if $X_0$ is an irreducible component of $X$, $Y_0=\phi(X_0)$ and $r$ is the minimum value of $\lambda(q)$ on $Y_0$, then \[
    \dim X_0 = \dim Y_0 + r.
    \]
\end{lemma}
\noindent
We note that if $Y$ is irreducible (which is the case if $X$ is irreducible), then upper semi-continuity of $\lambda(q)$ implies that its minimum value coincides with its generic value.

The following is a useful condition for the irreducibility of $X$.
\begin{lemma}[{\cite[Corollary~11.14]{Harris-93}}]
\label{lem:irred}
Suppose $X$ is a projective variety, $\phi:X\rightarrow \PP^n$ is a map and $Y=\phi(X)$.
    If $Y$ is irreducible and all fibers $\phi^{-1}(q)$ are irreducible of the same dimension, then $X$ is irreducible.
\end{lemma}

\subsection{Resultant variety}
We discuss the existence of the resultant variety, its irreducibility, and  its dimension.

Let $s$ and $n$ be two natural numbers with $s\geq n$ and assume $d_1,d_2,\dots,d_s\geq 1$.
We denote by $
\cP^{d_1,d_2,\dots,d_s}\coloneqq \KK[\xx]_{d_1}\times \KK[\xx]_{d_2}\times\dots\times \KK[\xx]_{d_s}
$ the vector space of $s$-tuples of polynomials in $n$ variables of degrees $d_1,d_2,\dots,d_s$.
\begin{proposition}[Resultant variety]
\label{prop:resultant_variety}
The set of tuples $\ff\in\cP^{d_1,\dots,d_s}$ with a common non-zero solution is an irreducible affine variety $\cR_{d_1,\dots,d_s}$ in $\cP^{d_1,\dots,d_s}$ of codimension $s-n+1$.
\end{proposition}
\begin{proof}
Let $\PP(\cP^{d_1,\dots,d_s})$ denote the projectivization of $\cP^{d_1,\dots,d_s}$ and consider the incidence variety \[
	\cI\coloneqq \{ ([\ff],[x])\in \PP(\cP^{d_1,\dots,d_s})\times\PP^{n-1}\mid f_1(x) = f_2(x) =\dots =f_s(x)=0 \}
	\] 
    and the double fibration 	
    \[
	\begin{tikzcd}
	& \cI \arrow[dl] \arrow[dr] & \\
	\PP(\cP^{d_1,\dots,d_s}) && \PP^{n-1}
	\end{tikzcd}
	\]
	given by the canonical projections. 
    We define $\cX\subset\PP(\cP^{d_1,\dots,d_s})$ to be the image of the first projection.
	Since $\cI$ is a projective variety, its image $\cX$ is closed, i.e., also a projective variety.

    We now prove that $\cX$ is irreducible of codimension $s-n+1$.
	Note that $\cI$ surjects onto $\PP^{n-1}$, since we can always find polynomials that vanish on any point $[x]\in\PP^{n-1}$. 
    Moreover, the fiber above $[x]$ is a linear subspace of $\PP(\cP^{d_1,\dots,d_s})$ of codimension $s$.
    By \cref{lem:irred}, $\cI$ is irreducible.
    Hence, \cref{lem:dimension} yields $\dim\cI=\dim\PP(\cP^{d_1,\dots,d_s})+\dim\PP^{n-1}-s$.
    Since $\cI$ is irreducible and $\cX$ is its image, $\cX$ is irreducible.
	Moreover, a generic system $\ff$ with a common solution has only finitely many solutions, i.e., the generic fiber dimension of the map $\cI\rightarrow\cX$ is zero.\footnote{Since $\cX$ is irreducible, the generic fiber dimension coincides with the minimum fiber dimension, hence it is enough to establish one system with finitely many solutions, e.g., one may take $f_i\coloneqq x_i^{d_i}$ for $i=1,\dots,n-1$ and $f_i\coloneqq x_1^{d_i}$ for $i\geq n$.}
    Thus, another application of \cref{lem:dimension} yields $\dim\cX = \dim \cI$, i.e., $\codim\cX=s-n+1$. 

    We just proved that $\cX\subset\PP(\cP^{d_1,\dots,d_s})$ is an irreducible projective variety of codimension $s-n+1$.
    Since $\cR_{d_1,\dots,d_s}\subset\cP^{d_1,\dots,d_s}$ is the affine cone over $\cX$, it is an irreducible affine variety of the same codimension.
\end{proof}

\subsection{The resultant ideal}
\label{sec:resultant-ideal}

We now introduce the universal coefficient ring attached to the degrees
$d_1,\ldots,d_s$. This gives a convenient way to define the resultant variety
and the notion of a resultant system.
We follow the development by Jouanolou \cite{Jouanolou-80,Jouanolou-91}.

For every $i\in [s]$ and every multi-index $\alpha \in\NN^n$ with
$|\alpha|=d_i$, let $y_{i,\alpha}$ be an independent variable.
The following ring, 
\[
	A_{\dd}
	\coloneqq
	\KK \big[y_{i,\alpha} \,:\, i\in [s] ,  |\alpha|=d_i\big],
	\qquad
	\text{ where } \ 
	\dd=(d_1,\ldots,d_s),
\]
is the universal ring of coefficients. 
We define the (generic) homogeneous forms
\[	
	F_i(\xx) \coloneqq
	\sum_{|\alpha|=d_i} y_{i,\alpha}\xx^\alpha
	\in A_{\dd}[x_1,\ldots,x_n],
	\qquad i \in [s].
\]
Thus $\Spec(A_{\dd})$ is the affine space parametrizing all
$s$-tuples of homogeneous forms of degrees $d_1,\ldots,d_s$.

A concrete set of polynomials (or a concrete polynomial system), 
say 
\[
	\Big\{ f_i(\xx)=\sum\nolimits_{|\alpha|=d_i} c_{i,\alpha}\xx^\alpha, \,:\,
		i \in [s] \Big\} \enspace,
\]
is obtained from the universal system by the specialization homomorphism
\[
	\sigma_{\ff} : A_{\dd}\longrightarrow \KK,
	\qquad \text{ where } \quad
	y_{i,\alpha}\longmapsto c_{i,\alpha}.
\]
In particular, for every polynomial $R\in A_{\dd}$, the value $R(\ff)$ means
$\sigma_{\ff}(R)$.

Let
\[
	C_{\dd}=A_{\dd}[x_1,\ldots,x_n],
	\qquad
	\mathfrak m= \ideal{x_1,\ldots,x_n} \subset C_{\dd},
	\qquad \text{ and } \ 
	I_{\dd}= \ideal{F_1,\ldots,F_s}\subset C_{\dd}.
\]
The condition that the specialized system has a non-zero solution is a projective
condition in the variables $\xx$. 
Hence, we obtain the corresponding elimination ideal
after saturating by $\mathfrak m$. We define the resultant ideal by
\[
	\mathfrak R_{\dd}
	\coloneqq
	\bigl(I_{\dd}:\mathfrak m^\infty\bigr)\cap A_{\dd}
	\enspace .
\]
\begin{remark}
The saturation by $\mathfrak m$ removes the spurious affine solution $\xx=0$ and keeps only the projective condition. Equivalently, $\mathfrak R_{\dd}$ consists of the polynomial conditions on the coefficients that vanish whenever the specialized system has a common point in projective space. 
\end{remark}
The zero set of $\mathfrak R_{\dd}$, over the algebraically 
closed field $\KK$,  is the \emph{resultant variety}, that is 
\[
\mathcal R_{\dd} =
	\cZ(\mathfrak R_{\dd}) =
\Big\{
	\ff=(f_1,\ldots,f_s)
	\in
	\prod_{i=1}^s \KK[\xx]_{d_i}
	\;:\;
	\cZ_{\PP}(f_1,\ldots,f_s)\neq \varnothing
\Big\}.
\]

This is the affine version of Jouanolou's universal construction of resultant
ideals: the coefficient ring parametrizes all systems and the resultant ideal cuts
out the locus of systems having a common projective zero
\cite{Jouanolou-80,Jouanolou-91}. It is also compatible with the usual
specialization property of resultants: the identities proved for the universal forms
specialize to identities for any concrete system of forms; see, for instance,
\cite[Chapter~6]{Buse-21}.

We can now define resultant systems intrinsically.

\begin{definition}[Resultant system]
\label{def:resultant-system-intrinsic}
A finite collection of polynomials $\{R_1,\ldots,R_N \} \subset A_{\dd}$
is called a resultant system for $s$ homogeneous forms of degrees
$d_1,\ldots,d_s$ in $n$ variables,
in the sense of \cref{def:resultant-system},  if
\[ 
	\sqrt{\ideal{R_1,\ldots,R_N}}= \sqrt{\mathfrak R_{\dd}}.
\]
Equivalently, for every system
$\ff=(f_1,\ldots,f_s)\in \prod_{i=1}^s\KK[\xx]_{d_i}$, one has
\[
	\cZ_{\PP}(f_1,\ldots,f_s)\neq\varnothing
	\quad\Longleftrightarrow\quad
	R_1(\ff)=\cdots=R_N(\ff)=0.
\]
\end{definition}

When $s=n$, the resultant variety is a hypersurface and the resultant ideal is
principal:
\[
	\mathfrak R_{d_1,\ldots,d_n} = 
	\ideal{\Res(F_1,\ldots,F_n)} \enspace.
\]
Thus, the classical Macaulay resultant is a resultant system with one polynomial.

When $s>n$, the resultant variety has higher codimension. In the generic
homogeneous case, this codimension is $s-n+1$. Therefore one should not expect a
single resultant polynomial. The natural object is the resultant ideal
$\mathfrak R_{\dd}$ and a resultant system is a finite set of polynomials
defining this ideal set-theoretically. The main purpose of this paper is to construct
 such systems of small cardinality by evaluating classical square resultants on suitable linear
combinations of the original forms.

\subsection{Resultant system as coefficients of Macaulay's resultant}
\label{sec:prelim-system}
In this section, we give a proof of \cref{thm:intro-abramov} using Perron's theorem, which states that we can obtain the common zero set of a homogeneous system 
from sufficiently generic square systems of linear combinations:

\begin{theorem}[{\cite[Satz~3]{Perron-41}}]
\label{thm:perron}
    Let $F_1,F_2,\dots,F_s\in\KK[x_1,\dots,x_n]_d$ be $s$ homogeneous polynomials in $n$ variables of degree $d$. 
    For a matrix $\Lambda\in\KK^{n\times s}$ consider the system of $n$ polynomials $G_1,\dots,G_n$, obtained as linear combinations of $F_1,\dots,F_s$: \[
    \begin{split}
    G_1 &\coloneqq \Lambda_{11} F_1 + \dots + \Lambda_{1s}F_s \\
    &\;\vdots\\
    G_n &\coloneqq \Lambda_{n1}F_1 +\dots+\Lambda_{ns}F_s
    \end{split}
    \] If $\Lambda$ is generic, then \[
    \cZ_{\PP}(F_1,\dots,F_s) = \cZ_{\PP}(G_1,\dots,G_n).
    \]
\end{theorem}
\noindent 
We remark that Perron only stated the existence of $\Lambda\in\KK^{n\times s}$, but the proof shows that a generic~$\Lambda$ has this property.
See also \cite{Chistov-Grigoriev-83} and \cite[Proposition~1]{Koiran-00} for constructive proofs of Perron's theorem and its applications.

\abramov*

\begin{proof}   
If $\cZ_{\PP}(f_1,\dots,f_s)\neq\varnothing$, then we have $\cZ_{\PP}(f_1,\dots,f_s)\subset\cZ_{\PP}(\sum_\ell \Lambda_{1\ell}f_\ell, \dots, \sum_{\ell}\Lambda_{n\ell} f_\ell)\neq \varnothing$ for every $\Lambda\in\KK^{n\times s}$. 
Hence, the resultant $\Res(\sum_\ell \Lambda_{1\ell}f_\ell, \dots, \sum_{\ell}\Lambda_{n\ell} f_\ell)$ identically vanishes as a polynomial in variables $\Lambda_{ij}, i\in [n], j\in [s]$. 
Hence, the coefficients $R_{\alpha_1,\dots,\alpha_n}$ of the expansion of this resultant vanish.

Conversely, suppose $\cZ_{\PP}(f_1,\dots,f_s)=\varnothing$.
    By Perron's \cref{thm:perron}, there exists a matrix $\Lambda\in\KK^{n\times s}$ such that $\cZ_{\PP}(\sum_{\ell=1}^s \Lambda_{1\ell}f_\ell, \dots,\sum_{\ell=1}^s \Lambda_{n\ell} f_\ell)=\varnothing$.
    Hence, the corresponding resultant $\Res(\sum_{\ell=1}^s \Lambda_{1\ell}f_\ell, \dots,\sum_{\ell=1}^s \Lambda_{n\ell} f_\ell)$ does not vanish.
    In particular, at least one of the coefficients $R_{\alpha_1,\dots,\alpha_n}$ does not vanish.
    This finishes the proof.
\end{proof}

\section{Proof of the main theorem}
\label{sec:main}

In this section, we prove 
\cref{thm:main-optimal}.
\subsection{Systems of linear equations}
\label{sec:linear-case}
First, we consider the linear case, $d=1$; \cref{cor:linear-resultant-system}.
We will prove it without using \cref{thm:main-optimal} as a preparation for the proof of the general case.

Given a matrix $C\in\KK^{s\times n}$, we consider the projective variety \[
\Sigma_C \coloneqq \{[\Lambda]\in\PP(\KK^{n\times s})\mid \det(\Lambda C)=0\}.
\] Note that $\Sigma_C$ is well-defined since $\det(\Lambda C)$ is a homogeneous polynomial.
\begin{lemma}
\label{lem:grassmannian}
$\Sigma_C$ depends only on the column space of $C$.
\end{lemma}
\begin{proof}
A matrix has the same column space as $C$ if and only if it equals $Cg$ for some $g\in \GL(n)$.
Then, $\det(\Lambda Cg)=\det(\Lambda C)\det(g)$ so $\Lambda \in \Sigma_C$ if and only if $\Lambda \in\Sigma_{Cg}$, i.e., $\Sigma_C=\Sigma_{Cg}$.
\end{proof}
\cref{lem:grassmannian} hints at considering the \emph{Grassmannian} \[
\GG(n,s) \coloneqq \{ L \mid L\text{ is an }n\text{ dimensional subspace of }\KK^s\}.
\] 
The Grassmannian $\GG(n,s)$ is an irreducible projective variety of dimension $n(s-n)$, see \cite[Lectures 6 and 11]{Harris-93}.
For a fixed subspace $L\in\GG(n,s)$, we define the variety \[
\Sigma_L\coloneqq \Sigma_{C_L}= \{ [\Lambda]\in\PP(\KK^{n\times s}) \mid \det(\Lambda C_L)=0 \},
\]
where $C_L\in\KK^{s\times n}$ denotes any matrix with column space $L$.
By \cref{lem:grassmannian}, $\Sigma_L$ only depends on $L$ and not on a specific choice of $C_L$.

\begin{lemma}
\label{lem:det-codim}
$\Sigma_L$ is an irreducible projective variety of codimension one.
\end{lemma}
\begin{proof}
If $L,L'\in\GG(n,s)$, then $C_L$ and $C_{L'}$ are $s\times n$ matrices of rank $n$.
Hence, there exists an invertible matrix $h\in\GL(s)$ with $C_L=hC_{L'}$ by Gaussian elimination.
This shows that the varieties $\Sigma_L$ and $\Sigma_{L'}$ are isomorphic via the map $[\Lambda]\mapsto [\Lambda h]$.
Thus, it is enough to prove the theorem for the block matrix $C_L\coloneqq\begin{pmatrix}
    I_n \\
    0
\end{pmatrix}$, whose column space $L$ is the span of the first $n$ standard basis vectors.
But, $\Sigma_{C_L}$ then consists of points $[\Lambda]$ such that the determinant of the $n\times n$-block obtained by taking the first $n$ columns of $\Lambda$ is zero.
Since the determinant is an irreducible polynomial, $\Sigma_{C_L}$ is an irreducible projective variety of codimension $1$.
\end{proof}

We can now prove \cref{cor:linear-resultant-system}.

\linearcor*

\begin{proof}
If $\rk C<n$, then $\det(\Lambda_i C)=0$ for every $i=1,\dots,N$.
Hence, we only need to prove that if $\rk C=n$, then there exists $i$ such that $\det(\Lambda_i C)\neq 0$.

For $k\in\ZZ_{>0}$ consider the incidence variety \[
    \cI_k \coloneqq \{(L,[\Lambda_1],\dots,[\Lambda_k])\in \GG(n,s)\times\PP(\KK^{n\times s})^{\times k}\mid \det(\Lambda_1 C_L)=\dots=\det(\Lambda_k C_L)=0\},
    \] where $C_L\in\KK^{s\times n}$ denotes any matrix with column space $L$.
    For a fixed $L\in\GG(n,s)$, the fiber above $L$ of the projection of $\cI_k$ onto $\GG(n,s)$ is isomorphic to the $k$-fold Cartesian product $\Sigma_L^{\times k}$, which is irreducible of codimension $k$ by \cref{lem:det-codim}. 
    Hence, \cref{lem:irred} implies that $\cI_k$ is irreducible and an application of
    \cref{lem:dimension} yields \[
    \dim\cI_k = \dim\GG(n,s)+\dim\left(\PP(\KK^{n\times s})^{\times k}\right)-k = n(s-n)+(sn-1)k-k.
    \] If $k> n(s-n)$, then $\dim\cI_k<(sn-1)k$.
    Hence, the image of $\cI_k$ is a proper closed subset of $\PP(\KK^{n\times s})^{\times k}$.

    This proves for $N\coloneqq n(s-n)+1$ that if $(\Lambda_1,\dots,\Lambda_N)$ is a generic point in $(\KK^{n\times s})^N$, then $([\Lambda_1],\dots,[\Lambda_N])$ is not contained in the image of $\cI_N$.
    Hence, for every $L\in\GG(n,s)$ there exists $i$ such that $\det(\Lambda_i C_L)\neq 0$. 
    Equivalently, for every matrix $C$ of rank $n$, there exists~$i$ such that $\det(\Lambda_i C)\neq 0$.
    This finishes the proof.
\end{proof}

\begin{remark}
    We note that the assumption $\mathrm{char}(\KK)=0$ is unnecessary for this proof.
\end{remark}

\subsection{Projective GIT quotient for systems of polynomials}
\label{sec:git}
The proof of \cref{thm:main-optimal} closely follows the outline of the proof in \cref{sec:linear-case}.
The main complication stems from finding the correct analog of the Grassmannian $\GG(n,s)$ for systems of $s$ homogeneous forms of degree $d$ in $n$ variables. 
We will achieve this by considering Mumford's \emph{projective GIT quotient} $\PP(\cP^{d,\dots,d})\sslash\SL(n)$ \cite{MFK-94}.
This is an irreducible projective variety of dimension $\binom{d+n-1}{n-1}s-n^2$ and contains a dense open subset that parametrizes systems with no common non-trivial solution up to a linear change of variables.
When $d=1$, we have $\PP(\cP^{d,\dots,d})\cong\PP(\KK^{s\times n})$ and the GIT quotient with respect to the action of $\SL(n)$ on $\KK^{s\times n}$ by right--multiplication coincides with the Grassmannian $\GG(n,s)\cong\PP(\KK^{s\times n})\sslash \SL(n)$.

The main theorem of this section is \cref{thm:git}, which states the existence and the properties of the projective GIT quotient for systems of homogeneous forms of arbitrary degree.
Our presentation of the subject follows the lecture notes by Victoria Hoskins on GIT quotients, we will use in particular \cite[Theorem~4.11]{Hoskins-12}.
The same results can also be found in \cite{MFK-94}.

We denote by $\SL(n)\coloneqq\{g\in\KK^{n\times n}\mid \det (g)=1\}$ the group of $n\times n$-matrices of unit determinant with entries in $\KK$. 
We view $\SL(n)$ as an $(n^2-1)$-dimensional affine variety in $\KK^{n\times n}$ whose defining equation is $\det(g)=1$.
It acts on $\KK[\xx]_d$ via $
(g\cdot f)(\xx)\coloneqq f(g^{-1}\cdot \xx)$.\footnote{The inversion makes this a left--action.} For example, \[
\begin{pmatrix}
    1 & 0\\
    1 & 1
\end{pmatrix} \cdot (x_1^2 + x_2^2) = (x_1-x_2)^2 + x_2^2 = x_1^2 -2 x_1 x_2 + 2 x_2^2.
\] The action of $\SL(n)$ on $\KK[\xx]_d$ extends to $s$-tuples in $\KK[\xx]_d^{s}$ by defining \[
g\cdot (f_1,f_2,\dots,f_s) \coloneqq (g\cdot f_1, g\cdot f_2,\dots, g\cdot f_s).
\]
Our first observation is that the existence of a non-zero solution stays invariant under the action of $\SL(n)$:
\begin{lemma}
\label{lem:orbit-solution}
    Let $\ff=(f_1,\dots,f_s)\in\cP^{d,\dots,d}$ and $g\in\SL(n)$.
    Then, $\ff$ has a non-zero solution if and only if $g\cdot \ff$ has a non-zero solution.
\end{lemma}
\begin{proof}
    $[x]\in\PP^{n-1}$ is a common solution of $f_1,\dots,f_s$ if and only if $[g x]\in\PP^{n-1}$ is a common solution of $g\cdot f_1,\dots,g\cdot f_s$.
\end{proof}
We say a set $S\subset\cP^{d,\dots,d}$ is $\SL(n)$-invariant if $g\cdot S\subset S$ for every $g\in\SL(n)$.
\cref{lem:orbit-solution} implies that the set of systems $\ff$ with a common non-zero solution, i.e., the resultant variety, is $\SL(n)$-invariant.
This suggests to \emph{mod out} the action of $\SL(n)$ on $\cP^{d,\dots,d}$ if we want to study the existence of non-zero solutions.
We achieve this via the projective GIT quotient $
\cX\coloneqq \PP(\cP^{d,\dots,d})\sslash\SL(n)
$, where $\PP(\cP^{d,\dots,d})\cong\PP^{\binom{d+n-1}{n-1}s-1}$ denotes the projectivization of $\cP^{d,\dots,d}$.

In order to construct the quotient, we need to consider the \emph{ring of invariants} of $\cP^{d,\dots,d}$:
Let $\KK[\cP^{d,\dots,d}]$ denote the affine coordinate ring of $\cP^{d,\dots,d}$.
The action of $\SL(n)$ on $\cP^{d,\dots,d}$ induces an action on $\KK[\cP^{d,\dots,d}]$ via \[
(g\cdot F)(\ff) = F(g^{-1}\cdot \ff), \quad F\in \KK[\cP^{d,\dots,d}], \; \ff\in\cP^{d,\dots,d}, \; g\in \SL(n).
\] We call $F$ an invariant if $g\cdot F=F$ for every $g\in\SL(n)$.
We denote the $\KK$-algebra of invariants by $\KK[\cP^{d,\dots,d}]^{\SL(n)}$.

Our aim in the remainder of this section is to prove the following theorem:
\begin{theorem}
\label{thm:git}
    There exist a projective variety $\cX$, an open subset $\PP(\cP^{d,\dots,d})^{ss}\subset\PP(\cP^{d,\dots,d})$ (called the semistable locus) and a surjective map $\phi:\PP(\cP^{d,\dots,d})^{ss}\rightarrow\cX$ with the following properties:
    \begin{enumerate}
        \item If $Z\subset \PP(\cP^{d,\dots,d})^{ss}$ is an $\SL(n)$-invariant closed set, then $\phi(Z)$ is closed.
        Moreover, a set $V\subset \cX$ is open if and only if $\phi^{-1}(V)$ is open.
        \item For $[\ff],[\ff']\in\PP(\cP^{d,\dots,d})^{ss}$ we have $\phi([\ff])=\phi([\ff'])$ if and only if the following holds: For every homogeneous invariant $F\in \KK[\cP^{d,\dots,d}]^{\SL(n)}$ we have $F(\ff)= 0 \iff F(\ff')=0$.
        \item If $\ff=(f_1,\dots,f_s)\in\cP^{d,\dots,d}$ has no common non-zero solution, then $[\ff]\in \PP(\cP^{d,\dots,d})^{ss}$.
        Moreover, the set of all $[\ff]$ with this property forms a non-empty $\SL(n)$-invariant open subset of $\PP(\cP^{d,\dots,d})^{ss}$.
        \item There exists a non-empty open dense subset $\cX^s\subset \cX$ such that $\phi:\phi^{-1}(\cX^s)\rightarrow\cX^s$ is an orbit space, i.e., for every $p\in\cX^s$, the fiber $\phi^{-1}(p)$ consists of a single $\SL(n)$-orbit.
        \item $\cX$ is an irreducible projective variety of dimension $\dim\cX = \binom{d+n-1}{n-1}s-n^2$.
    \end{enumerate}
\end{theorem}

By the Hilbert--Nagata finiteness theorem (\cite[\S 3.3]{Hoskins-12}), the invariant ring
$\KK[\cP^{d,\dots,d}]^{\SL(n)}$ is generated by a finite set of homogeneous polynomials in $\KK[\cP^{d,\dots,d}]$.
We denote by $\cX\coloneqq \Proj(\KK[\cP^{d,\dots,d}]^{\SL(n)})$ the projective variety that corresponds to $\KK[\cP^{d,\dots,d}]^{\SL(n)}$, see \cref{sec:varieties}.
The inclusion $\KK[\cP^{d,\dots,d}]^{\SL(n)}\xhookrightarrow{} \KK[\cP^{d,\dots,d}]$ induces a map $\phi:\PP(\cP^{d,\dots,d})^{ss}\rightarrow\cX$ on the locus of \emph{semistable} points defined as follows:
\begin{definition}
    A tuple $\ff\in\cP^{d,\dots,d}$ is called
    \begin{itemize}
        \item \emph{semistable} if there exists a non-constant homogeneous invariant $F\in \KK[\cP^{d,\dots,d}]^{\SL(n)}$ such that $F(\ff)\neq 0$.
    We denote by $(\cP^{d,\dots,d})^{ss}$ the set of all semistable tuples.
    \item It is called \emph{stable} if it is semistable, its orbit $\SL(n)\cdot\ff$ is Zariski closed, and its stabilizer \[
    \SL(n)_{\ff} \coloneqq \{g\in\SL(n)\mid \forall \ell=1,\dots,s, \; g\cdot f_\ell = f_\ell\}
    \] is finite. 
    We denote by $(\cP^{d,\dots,d})^{s}$ the set of all stable tuples. 
    \end{itemize}
\end{definition}

It easily follows from the definition that $(\cP^{d,\dots,d})^{ss}$ is an open subset of $\cP^{d,\dots,d}$ whose complement is the zero locus of all non-constant homogeneous invariants.\footnote{This locus is called the \emph{nullcone} of the action in the literature.}
It is known that the stable locus $(\cP^{d,\dots,d})^s$ is also an open subset of $\cP^{d,\dots,d}$, see \cite[Proposition~3.19]{Hoskins-12}.

If $\ff$ is semistable and $\lambda\in\KK^{\times}$ is a non-zero scalar, then $\lambda\ff$ is also semistable.
We will denote \[
\PP(\cP^{d,\dots,d})^{ss}\coloneqq \{[\ff]\in\PP(\cP^{d,\dots,d})\mid \ff\text{ is semistable}\},
\] which is an open subset of $\PP(\cP^{d,\dots,d})$.
We will now prove that the systems $\ff$ with no common solution are semistable.
To this end, we consider the invariants of $\cP^{d,\dots,d}$ that arise from Macaulay's resultant of linear combinations of the original system.
\begin{lemma}
\label{lem:resultant-invariance}
    Let $\Lambda\in\KK^{n\times s}$ be a matrix.
    Then the following is an $\SL(n)$-invariant: \[
    \Res(\Lambda\cdot\ff)\coloneqq \Res(\Lambda_{11}f_1+\dots+\Lambda_{1s}f_s,\,\dots,\, \Lambda_{n1}f_1+\dots+\Lambda_{ns}f_s)
    \] 
\end{lemma}
\begin{proof}
   Consider $h_i\coloneqq \sum_{\ell=1}^s \Lambda_{i\ell} f_\ell$.
    For $g\in\SL(n)$ we have $g\cdot h_i = g\cdot (\sum_{\ell}\Lambda_{i\ell} f_\ell)=\sum_{\ell}\Lambda_{i\ell}(g\cdot f_\ell)$.
    Hence,
    \[
    \Res\big(\sum_{\ell}\Lambda_{1\ell} (g\cdot f_\ell), \dots, \sum_{\ell}\Lambda_{n\ell} (g\cdot f_\ell)\big) = \Res(g\cdot h_1,\dots,g\cdot h_n) = \Res(h_1,\dots,h_n)=\Res(\Lambda\cdot\ff),
    \] where the second equality follows from the $\SL(n)$-invariance of Macaulay's resultant, see e.g. \cite[Theorem (3.5)]{Cox-Little-Oshea-05} or \cite[\S~5.13]{Jouanolou-91}.
    This finishes the proof.
\end{proof}
\begin{corollary}
\label{cor:semistable}
    Suppose $\ff\in\cP^{d,\dots,d}$ has no common non-zero solution.
    Then $\ff$ is semistable.
\end{corollary}
\begin{proof}
Suppose $\cZ_{\PP}(f_1,\dots,f_s)=\varnothing$.
    By \cref{thm:intro-abramov}, there exists $\Lambda\in\KK^{n\times s}$ with $\Res(\Lambda\cdot \ff)\neq 0$.
    By \cref{lem:resultant-invariance}, $\Res(\Lambda\cdot\ff)$ is an $\SL(n)$-invariant that does not vanish on $\ff$.
\end{proof}

We now prove that a generic element $\ff\in\cP^{d,\dots,d}$ is stable.
\begin{lemma}
\label{lem:stabilizer}
    Suppose $d\geq 1$.
    Then, the stabilizer of a generic $\ff\in\cP^{d,\dots,d}$ is finite.
\end{lemma}
\begin{proof}
The finiteness of the stabilizer is equivalent to $\dim\SL(n)_{\ff}=0$.
Now, $\dim\SL(n)_{\ff}$ is an upper semi-continuous function of $\ff$, see, e.g., \cite[Proposition~2.27]{Hoskins-12}.
Since $\cP^{d,\dots,d}$ is irreducible, it is enough to exhibit one tuple $\ff\in\cP^{d,\dots,d}$ with a finite stabilizer.

Consider the tuple $\ff=(x_1^d, x_2^d,\dots,x_n^d,0,\dots,0)\in\cP^{d,\dots,d}$. 
We claim that $\SL(n)_{\ff}$ is finite:
If $g\cdot \ff=\ff$, then for each $i=1,\dots,n$ we have $g\cdot x_i^d=x_i^d$. 
Hence, for $h=g^{-1}$ we have \[
\forall i=1,\dots,n,\quad \left(\sum_{j=1}^n h_{ji} x_j\right)^d = x_i^d.
\] By comparing the coefficient of $x_j^{d}$ in the left-- and the right--hand side we deduce that \[
h_{ji}^d = \begin{cases}
    0 & \text{ if }j\neq i\\
    1 & \text{ otherwise. }
\end{cases}
\] Thus, $h=g^{-1}$ is a diagonal matrix whose diagonal entries are $d$-th roots of unity.
This proves that $\SL(n)_{\ff}$ is finite.
\end{proof}

\begin{remark}
In fact, the stabilizer of a single generic polynomial ($s=1$) is already finite in the following cases:
Matsumura and Monsky proved that the stabilizer of a generic polynomial $f$ is finite if $n\geq 4$ and $d\geq 3$ in arbitrary characteristic, see \cite[Theorem~5]{Matsumura-Monsky-63}.
B{\"u}rgisser and Ikenmeyer extended this result to $n\geq 1$ and $d\geq 3$ when $\mathrm{char}(\KK)=0$, see \cite[Theorem~2.3]{BI-17}.
For $d=2$, the stabilizer of a generic quadratic form is isomorphic to the orthogonal group, and not finite.
However, a straightforward calculation shows that the stabilizer of a generic pair of quadratic forms is finite.
\end{remark}

\begin{corollary}
    A generic tuple $\ff\in\cP^{d,\dots,d}$ is stable.
\end{corollary}
\begin{proof}
A theorem by Luna \cite{Luna-73} (see also \cite[\S II.4.3D, Folgerung, p.142]{Kraft-84}) states for an action of a semisimple group (such as $\SL(n)$) that if the stabilizer of a generic element is finite, then the orbit of a generic element is also closed.
By \cref{lem:stabilizer}, this proves that a generic element is stable.
\end{proof}

We can now prove \cref{thm:git}.
\begin{proof}[Proof of \cref{thm:git}]
We define $\cX=\Proj \KK[\cP^{d,\dots,d}]^{\SL(n)}$.
By \cite[Theorem~4.11]{Hoskins-12}, the inclusion $\KK[\cP^{d,\dots,d}]^{\SL(n)}\xhookrightarrow{} \KK[\cP^{d,\dots,d}]$ induces a surjective map \[
\phi: \PP(\cP^{d,\dots,d})^{ss} \rightarrow \cX
\] on the semistable locus $\PP(\cP^{d,\dots,d})^{ss}$ which satisfies the first property, i.e., if $U\subset\PP(\cP^{d,\dots,d})^{ss}$ is an $\SL(n)$-invariant open (resp. closed) subset, then $\phi(U)$ is open (resp. closed).

Since $\cX$ is a projective variety, two points $[p],[p']\in\cX$ are equal if and only if the homogeneous prime ideals $\mathfrak{p}_{p}, \mathfrak{p}_{p'}$ that correspond to $p,p'$ are equal, i.e., we have $\{F\in\KK[\cX] \text{ homogeneous}\mid F(p)=0\}=\{F\in\KK[\cX]\text{ homogeneous}\mid F(p')=0\}$.
Since $\KK[\cX]=\KK[\cP^{d,\dots,d}]^{\SL(n)}$, the second property follows.

To prove the third property, we recall that if $\ff$ is a system with no common solution, then it is semistable (\cref{cor:semistable}).
Hence, $[\ff]\in \PP(\cP^{d,\dots,d})^{ss}$. 
Moreover, by \cref{prop:resultant_variety}, the systems with a common solution form a proper affine variety $\cR_{d,\dots,d}$ in $\cP^{d,\dots,d}$.
Hence, the set of points $[\ff]$ with $\ff$ not admitting a non-zero solution form an open subset $U$ of $\PP(\cP^{d,\dots,d})^{ss}$.
By \cref{lem:orbit-solution}, $U$ is $\SL(n)$-invariant.

By \cite[Theorem~4.11 (iii)]{Hoskins-12}, the map $\phi$ has the fourth property on the locus of stable elements, i.e., if we define $\cX^s=\phi(\PP(\cP^{d,\dots,d})^s)$, then $\phi^{-1}(\cX^s)=\PP(\cP^{d,\dots,d})^s$ and the induced map $\phi:\PP(\cP^{d,\dots,d})^s\rightarrow\cX^{s}$ is an orbit space.
Since $\PP(\cP^{d,\dots,d})^s$ is an $\SL(n)$-invariant open subset, $\cX^s=\phi(\PP(\cP^{d,\dots,d})^s)$ is an open subset.

Since $\cX$ is the image of an open subset of the projective space, it is irreducible.
Hence, every non-empty open subset of $\cX$ is dense.
In particular, $\cX^s$ is dense in $\cX$.

We now prove that $\dim\cX=\binom{d+n-1}{n-1}s-n^2$.
By \cite[Proposition~5.1]{Hoskins-12}, the stability of  $\ff$ is equivalent to the stability of $[\ff]$, i.e., the stabilizer $\SL_{[\ff]}$ is finite and its orbit $\SL(n)\cdot [\ff]$ is closed in $\cX^{s}$. 
This implies that the dimension of the orbit $\SL(n)\cdot [\ff]$ of a stable point $\ff$ equals $\dim\SL(n)-\dim\SL(n)_{[\ff]}=n^2-1$.
Hence, the dimension of the fiber of a point $p\in\cX^s$ equals $n^2-1$.
Since $\cX^s$ is dense in $\cX$, this implies that the fiber dimension of a generic point in $\cX$ is $n^2-1$.
An application of \cref{thm:harris-dim} to the quasi-projective variety $\PP(\cP^{d,\dots,d})^{ss}$ yields \[
\dim\cX = \dim \PP(\cP^{d,\dots,d}) - (n^2-1) =\binom{d+n-1}{n-1}s- 1 - (n^2-1)=\binom{d+n-1}{n-1}s-n^2.
\] 
\end{proof}

\subsection{Proof of the main theorem}
\label{sec:proof-main}
We can now prove \cref{thm:main-optimal}.
By \cref{thm:git}, we have a surjective map \[
\phi:\PP(\cP^{d,\dots,d})^{ss}\rightarrow \cX 
\] where $\cX=\PP(\cP^{d,\dots,d})\sslash \SL(n)$ is a projective variety and $\PP(\cP^{d,\dots,d})^{ss}$ denotes the semistable locus of $\PP(\cP^{d,\dots,d})$.

\mainthm*

\begin{proof}
It is easy to see that if $f_1,f_2,\dots,f_s$ have a common non-zero solution, then every resultant in \eqref{eq:resultant-system} vanishes. 
Hence, we need to prove for every $\ff=(f_1,\dots,f_s)\in\cP^{d,\dots,d}$ that if $\cZ_{\PP}(f_1,\dots,f_s)=\varnothing$ then there exists a polynomial in \eqref{eq:resultant-system} that does not vanish.

For a tuple $\ff=(f_1,\dots,f_s)\in\cP^{d,\dots,d}$ with $\cZ_{\PP}(f_1,\dots,f_s)=\varnothing$, we consider the map \[
\rho_{\ff} : \KK^{n\times s}\rightarrow \KK[\xx]_d^{\times n},\quad \Lambda\mapsto \Lambda \cdot \ff,
\] where $\Lambda\cdot \ff$ denotes the (square) system \[
\Lambda\cdot \ff = (\Lambda_{11}f_1+\dots+\Lambda_{1s}f_s, \dots, \Lambda_{n1}f_1+\dots+\Lambda_{ns}f_s).
\]
We denote by $\Sigma$ the resultant variety $\Sigma\coloneqq \{(g_1,\dots,g_n)\in\KK[\xx]_d^{\times n}\mid \Res(g_1,\dots,g_n)=0\}$ for square systems, where $\Res$ denotes Macaulay's resultant.
Then, $\Sigma$ has codimension $1$.
Since $f_1,\dots,f_s$ have no non-zero solution, Perron's \cref{thm:perron} implies that there exists $\Lambda\in\KK^{n\times s}$ such that $\cZ_{\PP}(\Lambda\cdot\ff)=\varnothing$.
This implies $\rho_{\ff}(\Lambda)\not\in\Sigma$ so the image of $\rho_{\ff}$ is not fully contained in $\Sigma$.
Thus, 
\begin{equation}
\label{eq:codim}
\forall \ff\in\cP^{d,\dots,d} \text{ with }\cZ_{\PP}(f_1,\dots,f_s)=\varnothing\text{ we have } \codim \rho_{\ff}^{-1}(\Sigma)=1.
\end{equation}

Let $U\subset\PP(\cP^{d,\dots,d})$ denote the set of all points $[\ff]$ with $\cZ_{\PP}(\ff)=\varnothing$, and set $V\coloneqq \phi(U)$.
We claim that $U=\phi^{-1}(V)$:
Assume that there exists $[\ff],[\ff']$ with $[\ff]\in U$ and $\phi([\ff])=\phi([\ff'])$.
We will prove that $[\ff']\in U$.
By \cref{thm:perron}, there exists a $\Lambda$ such that $\Res(\Lambda\cdot \ff)\neq 0$.
By \cref{lem:resultant-invariance}, $\cR(\Lambda\cdot\ff)$ is an invariant so $\cR(\Lambda\cdot\ff')\neq 0$ by \cref{thm:git}.
This proves that $\ff'$ has no common solution so $\ff'\in U$.

By \cref{thm:git} (1) and (3), $V$ is a non-empty open subset of $\cX$ so $V$ is a quasi-projective variety.
Moreover, $\overline{V}=\cX$ since $\cX$ is irreducible.
Consider the incidence variety \[
    \cI_k\coloneqq \left\{\, (\phi(\ff) ,\Lambda_1,\dots,\Lambda_k) \in V \times \left(\KK^{n\times s}\right)^{k} \mid \Res(\Lambda_i\cdot\ff)=0\right\}
    \] and the double fibration
        \[
	\begin{tikzcd}
	& \cI_k \arrow[dl] \arrow[dr] & \\
	V && \left(\KK^{n\times s}\right)^k
	\end{tikzcd}
	\]
given by the canonical projections.
    By \cref{lem:resultant-invariance}, $\Res(\Lambda_i\cdot\ff)$ is an $\SL(n)$-invariant so $\cI_k$ is well-defined by \cref{thm:git} (2).
    By \eqref{eq:codim}, the fibers of the first projection all have codimension $k$.
    Then, \cref{cor:quasi-proj-dim} yields \[
    \dim \cI_k = \dim\overline{V} + kns-k=\dim\cX+kns-k=N-1+kns-k.
    \] 
    Since $\cI_k$ is a quasi-projective variety, its image in $(\KK^{n\times s})^k$ is a constructible set.
    If $k>N-1$, then we have $\dim \cI_k < \dim \left((\KK^{n\times s})^{\times k}\right)=kns$, so the image is contained in a proper affine variety of codimension at least $1$.
    Hence, if $(\Lambda_1,\dots,\Lambda_{N})$ is generic, it is not contained in the image of $\cI_k$.
    This implies that it has the property that $\cZ_{\PP}(f_1,\dots,f_s)=\varnothing\implies\exists i, \Res(\Lambda_i\cdot \ff)\neq 0$.
    This finishes the proof.
\end{proof}

\section{Punctured resultant systems}
\label{sec:multiproj}
In this section, we prove \cref{thm:bivariate-explicit,thm:intro-explicit-multivariate}.
We will rely on the following useful lemma on linear combinations of vectors using coefficients from the Vandermonde matrix:

\begin{lemma}
\label{lem:vandermonde}
Suppose $m, s>0$ and $v_1,\dots,v_s\in\KK^m$ are $s$ vectors.
If $\lambda_1,\dots,\lambda_{s}\in\KK$ are $s$ distinct scalars, then the linear span of $v_1,\dots,v_s$ coincides with the linear span of the vectors
$w_1, \dots, w_s \in \KK^m$, where
 \[
\begin{split}
w_1 &= v_1 + \lambda_1 v_2 + \dots + \lambda_1^{s-1} v_s\\
& \qquad\vdots \\
w_s &= v_1 + \lambda_s v_2 + \dots + \lambda_s^{s-1} v_s.
\end{split}
\]
\end{lemma}
\begin{proof}
    Let $W\in\KK^{s\times m}$ be the matrix whose $i$-th row is $w_i^T$ and $V\in\KK^{s\times m}$ be the matrix whose $i$-th row is $v_i^T$.
    Then, we have $W = X V$ where $X=[\lambda_{i}^j]$ ($i=1,\dots,s, j=0,\dots,s-1$) is the Vandermonde matrix.
    Since $\lambda_1,\dots,\lambda_s$ are distinct, $X$ is invertible.
    Hence, $W$ and $V$ have the same row-span and this finishes the proof.
\end{proof}

\subsection{Explicit equations for bivariate systems}
We can now prove \cref{thm:bivariate-explicit}.

\puncturedbivariate*

\begin{proof}
    Set $N\coloneqq (s-2)d+1$.
    Let $\lambda_1,\lambda_2,\dots, \lambda_{N}\in \KK$ be an arbitrary collection of $N$ distinct scalars, and consider the following linear combinations of $f_1,\dots,f_{s-1}$:\[
    \begin{split}
        & f^{\lambda_1}\coloneqq f_1 + \lambda_1 f_2 + \dots + \lambda_1^{s-2} f_{s-1}\\
        &\qquad\vdots\\
        & f^{\lambda_N}\coloneqq f_1 + \lambda_N f_2 + \dots + \lambda_N^{s-2} f_{s-1}.
    \end{split}
    \]
   By \cref{lem:vandermonde}, any subcollection of $s-1$ of these polynomials have the same linear span as $f_1,\dots,f_{s-1}$.
    
    We now claim that the Sylvester resultants \begin{equation}
    \label{eq:bivariate-explicit}
    \Res(f^{\lambda_i}, f_s ), \quad i=1,\dots,N        
    \end{equation}
     form a punctured resultant system.
   If $f_1,\dots,f_s$ have a common non-zero solution $p$, then $f_s(p)=f^{\lambda_i}(p)=0$ so every polynomial in \eqref{eq:bivariate-explicit} vanishes.

    We now prove the opposite implication.
    Since $f_s$ is a non-zero polynomial of degree $d$, it vanishes on at most $d$ points in $\PP^1$, say $\{p_1,\dots,p_k\}\subset\PP^1$ with $k\leq d$.
    Since every resultant in \eqref{eq:bivariate-explicit} vanishes, for every $i=1,\dots,N$ there exists $j=1,\dots,k$ such that $f^{\lambda_i}(p_j)=0$.
    By pigeonhole principle, there exist $s-1$ linear combinations $f^{\lambda_{i_1}},\dots,f^{\lambda_{i_{s-1}}}$ that vanishes on the same root.
    However, by \cref{lem:vandermonde}, these polynomials have the same linear span as $f_1,\dots,f_{s-1}$, so $f_1(p_j)=\dots=f_{s-1}(p_j)=0$.
    Since $p_j$ was already a root of $f_s$, we have that $p_j$ is a common non-zero solution of the original system.
    This finishes the proof.
\end{proof}

\begin{remark}
    Suppose $N\geq s$.
    The \emph{spark} of a matrix $X\in\KK^{s\times N}$ is defined as the smallest integer $k$ such that there exists a set of $k$ columns which are linearly dependent. 
    \cref{lem:vandermonde} implies that the matrix $X_{ij}=[\lambda_j^{i}]$ with $i=0,\dots,s-1$ and  $j=1,\dots,N$ has $\mathrm{spark}(X)=s+1$.
    This is equivalent to non-vanishing of all maximal minors of $X$.
    From the proof of \cref{thm:bivariate-explicit}, it is apparent that any matrix $X$ with full spark can be used to construct a punctured resultant system.
    An example is the block matrix $X=\begin{pmatrix}
        I_s & C
    \end{pmatrix}$
    where $C$ is the \emph{Cauchy matrix}, i.e., $C_{ij}=\frac{1}{x_i-y_j}$ where $x_1,\dots,x_s,y_1,\dots,y_{N-s}$ are distinct scalars.
    Among all matrices with full spark, $X$ has the largest possible number of zero entries.
    Indeed, if an $s\times N$-matrix has more than $s^2-s$ zeros, then it must have a row with $s$ zeros and the corresponding maximal minor vanishes.
\end{remark}

\begin{example}
    We consider the first non-trivial example of $s=3$ quadratic forms in two variables \[
    \begin{split}
            f_1 &= a_0 x^2 + a_1 xy + a_2 y^2\\
            f_2 &= b_0 x^2 + b_1 xy + b_2 y^2\\
            f_3 &= c_0 x^2 + c_1 xy + c_2 y^2.
    \end{split}
    \]
    From the proof of \cref{thm:bivariate-explicit}, we see that the resultants \[
    R_1\coloneqq \Res(f_1, f_3), \; R_2\coloneqq\Res(f_1+f_2,f_3),\; R_3\coloneqq \Res(f_1 - f_2, f_3)
    \] are a punctured resultant system.
    We note that these polynomials do not form a resultant system in the ordinary sense since they always vanish if $f_3=0$.
    A sparser system is $\Res(f_1,f_3), \Res(f_2,f_3)$ and $\Res(f_1+f_2,f_3)$. 
    These polynomials form a punctured resultant system since any pair of $f_1,f_2,f_1+f_2$ has the same linear span as $f_1,f_2$.
\end{example}

\subsection{Explicit equations for general systems}

There is a strategy to generalize the construction for the bivariate case to $n>2$:
If the zero locus of the last $n-1$ polynomials $f_{s-n+1},\dots,f_{s}$ is finite, then it contains at most $d^{n-1}$ points by the B\'{e}zout theorem. 
Hence, if we enforce the vanishing of $(s-n)d^{n-1}+1$ resultants of the form \[
\Res(f^{\lambda_i},f_{s-n+1},\dots,f_s) = 0, \quad i=1,2,\dots,(s-n)d^{n-1}+1
\] for $f^{\lambda_i}\coloneqq f_1 + \lambda_i f_2 + \dots + \lambda_i^{s-n-1}f_{s-n}$, then by pigeonhole principle, $s-n+1$ of the linear combinations necessarily vanish on the same root.
By a similar argument as in the proof of \cref{thm:bivariate-explicit}, we then deduce that $f_1,\dots,f_s$ share a common non-zero root. 

Unfortunately, this strategy fails when the zero locus of $f_{s-n+1},\dots,f_s$ is not finite.
Instead, we will show that if we consider sufficiently many linear combinations, we can enforce that a collection of linear combinations have a zero dimensional solution set, and then we can apply this strategy.

We recall that a variety is said to be pure $m$-dimensional if all of its irreducible components are $m$-dimensional.

\begin{lemma}
\label{prop:dimension-drop}
    Suppose $C\subset\PP^{n-1}$ is pure $m$-dimensional of degree~$D$.
    Let \[
    f^\lambda \coloneqq f_1 + \lambda f_2 + \dots + \lambda^{q-1} f_q
    \] for $q$ forms $f_1,f_2,\dots,f_q$ and $\lambda\in\KK$.
    Suppose:
    \begin{enumerate}
        \item $\dim\left(C\cap\cZ_{\PP}(f_1,f_2,\dots,f_q)\right)<m$;
        \item $C\subset\cZ_{\PP}(f^{\tau_1},f^{\tau_2},\dots,f^{\tau_r})$ for distinct scalars $\tau_1,\dots,\tau_r$;
        \item $S\subset\KK$ is disjoint from $\{\tau_1,\dots,\tau_r\}$ and satisfies \[
        |S|\geq (q-r-1)D+1.
        \]
    \end{enumerate}
    Then some $\lambda\in S$ satisfies \[
    \dim\left(\, C\cap\cZ_{\PP}(f^\lambda) \,\right) = m-1.
    \]
\end{lemma}
\begin{proof}
    Let $C=C_1 \cup C_2 \cup\dots \cup C_c$ be the irreducible decomposition of $C$. 
    Since $\deg C = \sum_\ell \deg C_\ell = D$, we have $c\leq D$.

    To reach a contradiction, assume that for every $\lambda\in S$ there exists $1\leq \ell\leq c$ such that $f^\lambda$ vanishes identically on $C_\ell$.
    By pigeonhole principle, there exist at least $q-r$ scalars $\lambda_1,\dots,\lambda_{q-r}$ such that $f^{\lambda_1},\dots,f^{\lambda_{q-r}}$ vanish on the same component, say $C_\ell$.
    Together with the linear combinations $\tau_1,\dots,\tau_{r}$, we deduce that $q$ distinct linear combinations vanish on $C_\ell$.
    By \cref{lem:vandermonde}, these linear combinations have the same zero locus as $f_1,\dots,f_q$.
    Hence, $C_\ell\subset\cZ_{\PP}(f_1,\dots,f_q)$.
    This contradicts the assumption that $\dim\left(C\cap \cZ(f_1,\dots,f_q)\right)<m$.

    We proved that there exists a linear combination $f^{\lambda}$ that does not vanish on any component $C_j$, $j=1,\dots,c$.
    Hence, we have $\dim(C\cap \cZ(f^{\lambda}))=m-1$ and this finishes the proof.
\end{proof}

We can now prove \cref{thm:intro-explicit-multivariate}.
\begin{theorem}
\label{thm:explicit-multivariate}
Suppose $S_1,S_2,\dots,S_{n-1}\subset\KK$ are pairwise disjoint sets of scalars of cardinalities \[
\forall \ell=1,2,\dots,n-1, \; |S_{\ell}| = (s-\ell-1)d^{\ell}+1.
\] Then the following collection of polynomials is a punctured resultant system: \begin{equation}
\label{eq:explicit-system-multi}
\Res(f^{\lambda_1}, f^{\lambda_2},\dots, f^{\lambda_{n-1}}, f_s),\quad (\lambda_1,\dots,\lambda_{n-1})\in S_1 \times S_2\times\dots\times S_{n-1},
\end{equation} where we denote $f^{\lambda_i}\coloneqq f_1 + \lambda_i f_2 +\dots + \lambda_i^{s-2} f_{s-1}$.

In particular, there exists a punctured resultant system of cardinality \[
|S_1\times\dots\times S_{n-1}| = \prod_{\ell=1}^{n-1} (s-\ell-1) d^{\ell} + 1.
\] 
\end{theorem}
\begin{proof}
We need to prove that if every resultant in \eqref{eq:explicit-system-multi} vanishes, then $f_1,\dots,f_s$ have a common non-zero solution.
Since $f_s\neq 0$, the hypersurface $C_0\coloneqq \cZ_{\PP}(f_s)\subset\PP^{n-1}$ has dimension $n-2$ and its degree is bounded by $d$. 
If $\cZ(f_1,\dots,f_{s-1})\cap C_0\neq\varnothing$, then we are done.
Otherwise, an application of \cref{prop:dimension-drop} to $C_0$ and to the polynomials $f_1,\dots,f_{s-1}$ with $q=s-1$ and $r=0$ implies that there exists $\lambda^{(1)}\in S_1$ such that \[
\dim\left(\cZ(f_s,f^{\lambda^{(1)}})\right) = n-3.
    \] 
    Hence, the intersection $C_1\coloneqq\cZ_{\PP}(f_s,f^{\lambda^{(1)}})$ is a pure $(n-3)$-dimensional projective algebraic variety and its degree is bounded by $d^2$ by B\'{e}zout's theorem.
    If $C_1$ intersects $\cZ_{\PP}(f_1,\dots,f_{s-1})$, then we are done. 
    Otherwise, another application of \cref{prop:dimension-drop} to $C_1$ with $q=s-1$ and $r=1$ yields a $\lambda^{(2)}\in S_2$ such that \[
    C_2\coloneqq \cZ_{\PP}(f_s,f^{\lambda^{(1)}},f^{\lambda^{(2)}}),\quad \dim C_2 = n-4.
    \] By iterating this process, we can find points $\lambda^{(i)}\in S_i$ such that \[
    C_{n-2}\coloneqq \cZ_{\PP}(f_s,f^{\lambda^{(1)}},f^{\lambda^{(2)}},\dots,f^{\lambda^{(n-2)}}),\quad \dim C_{n-2} = 0.
    \] By B\'{e}zout's theorem, the cardinality of $C_{n-2}$ is at most $\leq d^{n-1}$.

    By the assumption, the resultant \[
    \Res(f^{\mu},f^{\lambda^{(1)}},f^{\lambda^{(2)}},\dots,f^{\lambda^{(n-2)}},f_s) = 0
    \] vanishes for every $\mu\in S_{n-1}$.
    Hence, for every $\mu\in S_{n-1}$ there exists a point $p\in C_{n-2}$ such that $f^{\mu}$ vanishes on $p$.
    By pigeonhole principle, there exist $\mu_1,\dots,\mu_{s-n+1}\in S_{n-1}$ such that $f^{\mu_1},\dots,f^{\mu_{s-n+1}}$ vanish on the same point $p\in C_{n-2}$. 
    By \cref{lem:vandermonde},
    the polynomials $f^{\mu_1},\dots,f^{\mu_{s-n+1}}, f^{\lambda^{(1)}},\dots,f^{\lambda^{(n-2)}}$ have the same linear span as $f_1,\dots,f_{s-1}$.
    Hence, $p$ is in the common zero locus of $f_1,\dots,f_{s-1},f_s$.
    This finishes the proof.
\end{proof}

\section{Conclusions and outlook}
We have constructed small resultant systems by evaluating the classical Macaulay's resultant on linear combinations of the input forms.
For homogeneous systems of equal degree, this yields a resultant system of cardinality $\binom{d+n-1}{n-1}s-n^2+1$; in the linear case, this specializes to an optimal set-theoretic description of the corresponding determinantal variety.
By assuming that none of the input forms is zero, we have also obtained fully explicit punctured resultant systems, consisting of $(s-2)d+1$ polynomials in the bivariate case and polynomially number of equations for every fixed number of variables in general.
A natural direction for the future work is making the choice of the matrices in \cref{thm:main-optimal} explicit.

\section*{Acknowledgements}
L. D. was supported by the European Union (ERC Grant SYMOPTIC, 101040907) and by the Deutsche Forschungsgemeinschaft (DFG, German Research Foundation, 556164098).
ET is  partially supported by the PGMO grant SOAP, 
the PHC PROCOPE project ``Quantum  games and polynomial optimization'',
and ANR PRC ZADyG (ANR-25-CE48-7058).

\bibliographystyle{amsalpha}
\bibliography{references}

@ARTICLE{Abramov-13,
  title     = {A resultant system as the set of coefficients of a single resultant},
  author    = {Abramov, Ya. V.},
  journal   = {Functional Analysis and Its Applications},
  publisher = {Springer Science and Business Media LLC},
  volume    =  47,
  number    =  3,
  pages     = {233--237},
  month     =  jul,
  year      =  {2013},
  copyright = "https://www.springernature.com/gp/researchers/text-and-data-mining"
}

@article{Encarnacion-98,
	title = {An Efficient Method for Computing Resultant Systems},
	volume = {9},
	issn = {0938-1279, 1432-0622},
	url = {http://link.springer.com/10.1007/s002000050105},
	doi = {10.1007/s002000050105},
	number = {3},
	urldate = {2023-07-13},
	journal = {Applicable Algebra in Engineering, Communication and Computing},
	author = {Encarnaci{\'o}n, M. J.},
	month = nov,
	year = {1998},
	pages = {243--245}
}

@article{Lyubeznik-95,
title = {Minimal Resultant Systems},
journal = {Journal of Algebra},
volume = {177},
number = {2},
pages = {612-616},
year = {1995},
issn = {0021-8693},
doi = {https://doi.org/10.1006/jabr.1995.1315},
url = {https://www.sciencedirect.com/science/article/pii/S0021869385713154},
author = {Lyubeznik, G.}
}

@article{CMSS-25,
title={Polynomial systems admitting a simultaneous solution}, 
author={Conner, A. and Micha{\l}ek, M. and Schindler, M. and Szendr{\H{o}}i, B.},
journal = {Journal of Algebra},
volume = {667},
pages = {412--424},
year = {2025},
issn = {0021-8693},
doi = {https://doi.org/10.1016/j.jalgebra.2024.12.015},
url = {https://www.sciencedirect.com/science/article/pii/S0021869324006902}
}

@article{Kakie-76,
 ISSN = {00029939, 10886826},
 URL = {http://www.jstor.org/stable/2040735},
 author = {Kaki{\'{e}}, K.},
 journal = {Proceedings of the American Mathematical Society},
 number = {1},
 pages = {1--7},
 publisher = {American Mathematical Society},
 title = {The Resultant of Several Homogeneous Polynomials in Two Indeterminates},
 urldate = {2026-05-26},
 volume = {54},
 year = {1976}
}

@article{Orsinger-51,
author = {Orsinger, H.},
title = {Zur {K}onstruktion von {T}r{\"{a}}gheitsformen als {K}oeffizienten algebraischer {G}leichungen},
journal = {Mathematische Nachrichten},
volume = {5},
number = {6},
pages = {355-370},
doi = {https://doi.org/10.1002/mana.19510050604},
url = {https://onlinelibrary.wiley.com/doi/abs/10.1002/mana.19510050604},
eprint = {https://onlinelibrary.wiley.com/doi/pdf/10.1002/mana.19510050604},
year = {1951}
}

@article{Bruns-Schwanzl-90,
    author = {Bruns, W. and Schw{\"{a}}nzl, R.},
    title = {The Number of Equations Defining a Determinantal Variety},
    journal = {Bulletin of the London Mathematical Society},
    volume = {22},
    number = {5},
    pages = {439--445},
    year = {1990},
    month = sep,
    issn = {0024-6093},
    doi = {10.1112/blms/22.5.439},
    url = {https://doi.org/10.1112/blms/22.5.439}
}

@InProceedings{Bruns-89,
author={Bruns, W.},
editor={Hochster, M.
and Huneke, C.
and Sally, J. D.},
title={Additions to the Theory of Algebras with Straightening Law},
booktitle={Commutative Algebra},
year={1989},
publisher={Springer New York},
address={New York, NY},
pages={111--138},
isbn={978-1-4612-3660-3}
}

@book{Kapferer-29,
author  = {Kapferer, H.},
title   = {{\"{U}}ber {R}esultanten und {R}esultanten-{S}ysteme},
year    = {1929},
address = {M\"{u}nchen},
series  = {Sitzungsberichte der mathematisch-naturwissenschaftlichen Abteilung der Bayerischen Akademie der Wissenschaften},
volume  = {1929,12},
publisher = {{V}erlag der {B}ayerischen {A}kademie der {W}issenschaften},
pages = {179--200}
}

@book{vdWaerden-40,
  author    = {van der Waerden, B. L.},
  title     = {Modern Algebra},
  volume     = {2},
  translator = {Theodore J. Benac},
  publisher  = {Frederick Ungar Publishing Co.},
  address    = {New York},
  year       = {1940},
  note        = {Translated from the second revised German edition}
}

@book{Yap-99,
  title     = {Fundamental {P}roblems of {A}lgorithmic {A}lgebra},
  author    = "Yap, C. K.",
  publisher = {Oxford University Press},
  month     =  dec,
  year      =  {1999},
  address   = {New York, NY}
}

@article{Perron-41,
  author          = {Perron, O.},
  title           = {Beweis und {V}ersch{\"a}rfung eines {S}atzes von {K}ronecker},
  journal         = {Mathematische Annalen},
  year            = {1941},
  month           = dec,
  day             = {01},
  volume          = {118},
  number          = {1},
  pages           = {441--448},
  issn            = {1432-1807},
  doi             = {10.1007/BF01487380},
  url             = {https://doi.org/10.1007/BF01487380}
}

@article{Macaulay-02,
author = {Macaulay, F. S.},
title = {Some {F}ormul{\ae} in {E}limination},
journal = {Proceedings of the London Mathematical Society},
volume = {s1-35},
number = {1},
pages = {3--27},
doi = {https://doi.org/10.1112/plms/s1-35.1.3},
url = {https://londmathsoc.onlinelibrary.wiley.com/doi/abs/10.1112/plms/s1-35.1.3},
eprint = {https://londmathsoc.onlinelibrary.wiley.com/doi/pdf/10.1112/plms/s1-35.1.3},
year = {1902}
}

@book{Cox-Little-Oshea-05,
  title     = "Using algebraic geometry",
  author    = {Cox, D. A. and Little, J. B. and O'Shea, D.},
  publisher = {Springer},
  series    = {Graduate Texts in Mathematics},
  edition   =  {2nd},
  month     =  mar,
  year      =  {2005},
  address   = {New York, NY}
}

@book{Hodge-Pedoe-94, 
place={Cambridge}, 
series={Cambridge Mathematical Library}, 
title={Methods of Algebraic Geometry}, 
publisher={Cambridge University Press}, 
author={Hodge, W. V. D. and Pedoe, D.}, 
year={1994}, 
collection={Cambridge Mathematical Library}
}

@book{Macaulay-16,
  author    = {Macaulay, F. S.},
  title     = {The Algebraic Theory of Modular Systems},
  publisher = {Cambridge University Press},
  address   = {Cambridge},
  year      = {1916}
}

@article{Sylvester-40,
  author  = {Sylvester, J. J.},
  title   = {A method of determining by mere inspection the derivatives from two equations of any degree},
  journal = {The London, Edinburgh, and Dublin Philosophical Magazine and Journal of Science},
  volume  = {16},
  year    = {1840},
  pages   = {132--135}
}

@book{Harris-93,
  series          = {Graduate texts in mathematics},
  publisher       = {Springer},
  isbn            = 0387977163,
  year            = {1993},
  title           = {Algebraic geometry : a first course},
  edition         = {Corr. 2nd printing},
  address         = {New York u.a.},
  author          = {Harris, J.},
}

@book{Lang-12,
  title     = {Algebra},
  author    = {Lang, S.},
  publisher = {Springer},
  series    = {Graduate Texts in Mathematics},
  edition   =  {3rd},
  month     =  nov,
  year      =  {2002},
  address   = {New York, NY}
}

@article{Nie-10,
  title={An exact {J}acobian {SDP} relaxation for polynomial optimization},
  author={Nie, J.},
  journal={Mathematical Programming},
  year={2013},
  volume={137},
  pages={225--255}
}

@book{MFK-94,
    author = {Mumford, D. and Fogarty, J. and Kirwan, F.},
     title = {Geometric invariant theory},
    series = {Ergebnisse der Mathematik und ihrer Grenzgebiete (2) [Results
              in Mathematics and Related Areas (2)]},
    volume = {34},
   edition = {3rd},
publisher = {Springer-Verlag, Berlin},
      year = {1994},
       DOI = {10.1007/978-3-642-57916-5},
       URL = {https://doi.org/10.1007/978-3-642-57916-5},
}

@article{BI-17,
title = {Fundamental invariants of orbit closures},
journal = {Journal of Algebra},
volume = {477},
pages = {390--434},
year = {2017},
issn = {0021-8693},
doi = {https://doi.org/10.1016/j.jalgebra.2016.12.035},
url = {https://www.sciencedirect.com/science/article/pii/S0021869317300327},
author = {B\"{u}rgisser, P. and Ikenmeyer, C.}
}

@book{Shafarevich-13,
  title     = {Basic algebraic geometry 1},
  author    = {Shafarevich, I. R.},
  publisher = {Springer},
  edition   =  {3rd},
  month     =  aug,
  year      =  {2013},
  address   = "Berlin, Germany",
  copyright = "https://www.springernature.com/gp/researchers/text-and-data-mining"
}

@misc{Hoskins-12,
title = {Geometric invariant theory and symplectic quotients},
author = {Hoskins, V.},
year = {2012},
howpublished = {\url{https://userpage.fu-berlin.de/hoskins/GITnotes.pdf}},
note = "[Online; accessed 10-June-2026]"
}

@incollection{Luna-73,
     author = {Luna, D.},
     title = {Slices {\'{e}}tales},
     booktitle = {Sur les groupes alg{\'e}briques},
     series = {Bulletin de la Soci{\'e}t{\'e} math{\'e}matique de France. M{\'e}moire},
     pages = {81--105},
     year = {1973},
     publisher = {Soci{\'e}t{\'e} math{\'e}matique de France},
     number = {33},
     doi = {10.24033/msmf.110},
     zbl = {0286.14014},
     url = {https://www.numdam.org/articles/10.24033/msmf.110/}
}

@book{Kraft-84,
  title     = {Geometrische {M}ethoden in der {I}nvariantentheorie},
  author    = {Kraft, H.},
  publisher = {Springer Vieweg},
  series    = {Aspects of Mathematics},
  month     =  jan,
  edition = {1st},
  year      =  {1984},
  address   = {Wiesbaden, Germany}
}

@book{Cox-Little-Oshea-25,
  title     = {Ideals, varieties, and algorithms},
  author    = {Cox, D. A. and Little, J. B. and O'Shea, D.},
  publisher = {Springer International Publishing},
  series    = {Undergraduate Texts in Mathematics},
  edition   =  {5th},
  month     =  aug,
  year      =  {2025},
  address   = {Cham, Switzerland},
  copyright = "https://www.springernature.com/gp/researchers/text-and-data-mining"
}

@article{Jouanolou-91,
title = {Le formalisme du r{\'e}sultant},
journal = {Advances in Mathematics},
volume = {90},
number = {2},
pages = {117--263},
year = {1991},
issn = {0001-8708},
doi = {https://doi.org/10.1016/0001-8708(91)90031-2},
url = {https://www.sciencedirect.com/science/article/pii/0001870891900312},
author = {Jouanolou, J.-P.}
}

@article{Matsumura-Monsky-63,
  title     = {On the automorphisms of hypersurfaces},
  author    = {Matsumura, H. and Monsky, P.},
  journal   = {Kyoto J. Math.},
  publisher = {Duke University Press},
  volume    =  {3},
  number    =  {3},
  pages     = {347--361},
  month     =  jan,
  year      =  {1963}
}

@techreport{Chistov-Grigoriev-83,
  title={Sub--exponential time solving systems of algebraic equations {I}, {II}},
  author={Chistov, A. L. and Grigoriev, D. Yu.},
  year={1983},
  institution={LOMI preprint E-9-83, E-10-83, Steklov Institute, Leningrad}
}

@inproceedings{Koiran-00, 
author = {Koiran, P.}, 
title = {Circuits versus Trees in Algebraic Complexity}, 
year = {2000}, 
isbn = {3540671412}, 
publisher = {Springer-Verlag}, 
address = {Berlin, Heidelberg}, booktitle = {Proceedings of the 17th Annual Symposium on Theoretical Aspects of Computer Science}, 
pages = {35--52}, 
numpages = {18}, 
series = {STACS '00} 
}

@article{Jouanolou-80,
  author  = {Jouanolou, J.-P.},
  title   = {Id{\'e}aux r{\'e}sultants},
  journal = {Advances in Mathematics},
  volume  = {37},
  year    = {1980},
  pages   = {212--238}
}

@misc{Buse-21,
  author = {Bus{\'e}, Laurent},
  title  = {Computational Algebraic Geometry},
  note   = {Master 2 lecture notes, Universit{\'e} C{\^o}te d'Azur and Inria},
  year   = {2021},
  url    = {https://inria.hal.science/hal-03708355}
}

\end{document}